\documentclass[a4paper, 10pt, reqno]{amsart}
\usepackage{amsmath, amscd}
\usepackage{amssymb, color}
\usepackage[colorlinks]{hyperref}
\definecolor{linkblue}{RGB}{1,1,190}
\definecolor{citered}{RGB}{190,1,1}
\hypersetup{
	linkcolor=linkblue,
	urlcolor=linkblue,
	citecolor=citered
}
\usepackage{comment}

\makeatletter
\def\namedlabel#1#2{\begingroup
    #2%
    \def\@currentlabel{#2}%
    \phantomsection\label{#1}\endgroup
}
\makeatother

\def\doi#1{{\small\href{https://doi.org/#1}{\path{doi:#1}}}}
\def\arxiv#1{{\small\href{http://www.arxiv.org/abs/#1}{\path{arXiv:#1}}}}
\def\url#1{{\small\href{#1}{\path{#1}}}}

\theoremstyle{plain}
\newtheorem{theorem}{\bf Theorem}[section]
\newtheorem{proposition}[theorem]{\bf Proposition}
\newtheorem{lemma}[theorem]{\bf Lemma}

\theoremstyle{definition}
\newtheorem{example}[theorem]{\bf Example}

\newtheorem{remark}[theorem]{\bf Remark}

\newcommand{\N}{\mathbb N}

\newcommand{\bdot}{\boldsymbol{\cdot}}

 \DeclareMathOperator{\ord}{ord}

 \DeclareMathOperator{\supp}{supp}

 \DeclareMathOperator{\Syl}{Syl}

\newcommand{\red}{{\text{\rm red}}}
\renewcommand{\t}{\, | \,}

\numberwithin{equation}{section}

\begin{document}

\title{The catenary degree of monoids of product-one sequences}

\author{Jun Seok Oh}
\address{Department of Mathematics Education, Jeju National University, Jeju 63243, Republic of Korea}
\email{junseok.oh@jejunu.ac.kr}


\thanks{This work was supported by the research grant of Jeju National University in 2026}

\begin{abstract}
Let $G$ be a (multiplicatively written) finite group.
A sequence over $G$ is a finite collection of terms from $G$, where repetition is allowed and the order is disregarded.
A product-one sequence is a sequence whose terms can be ordered such that their product in $G$ equals the identity element of $G$.
The set $\mathcal B (G)$ of all product-one sequences over $G$, endowed with the concatenation of sequences as the operation, is a finitely generated C-monoid; in particular, it is atomic, i.e., every non-unit element can be written as a finite product of atoms.
The study of $\mathcal B (G)$ is of fundamental importance, as its combinatorial, algebraic, and arithmetic properties play a crucial role across various branches of mathematics, most notably in invariant theory and factorization theory.
While the arithmetic of the monoid $\mathcal B (G)$ is well understood in the abelian setting (in which case $\mathcal B (G)$ is a Krull monoid), little is known in the non-abelian setting because of the substantial structural complexity involved.
In this paper, we study the arithmetic invariants of the monoid $\mathcal B (G)$ for non-abelian groups, focusing in particular on the catenary degree.
The catenary degree $\mathsf c (G)$ of the monoid $\mathcal B (G)$ is defined as the smallest integer $N$ such that any two factorizations of an element $S \in \mathcal B (G)$ can be concatenated by a chain of factorizations in which adjacent steps differ by replacing at most $N$ atoms.
Extending the methods from arithmetic combinatorics to the non-abelian setting, we explicitly characterize all finite groups with catenary degree at most 3, and we investigate an infinite class of finite groups whose monoids of product-one sequences are seminormal and possess well-behaved arithmetic structures.
Furthermore, we show that a specific non-abelian group in this class has catenary degree 4.
\end{abstract}

\maketitle


\section{Introduction} \label{1}
\medskip

The study of sequences over a finite group $G$, along with the algebraic structures they generate, stands as a classical and central topic in arithmetic combinatorics with profound implications in various branches of mathematics.
In particular, the monoid of product-one sequences over $G$, denoted by $\mathcal B (G)$, has deep connections to several fields, most notably invariant theory and factorization theory.
In invariant theory, evaluating the Noether number $\boldsymbol{\beta} (G)$ and the separating Noether number $\boldsymbol{\beta}_{\textnormal{sep}} (G)$ of polynomial invariant rings is a fundamental problem.
It has been established that these invariants are strictly governed by the combinatorial invariants of $\mathcal B (G)$, making the study of product-one sequences indispensable (see \cite{Sc91,Cz-Do-Ge16,Cz-Do-Sz18,Ha-Zh19,Do-Sc24,Do-Sc25,Sc25a,Sc25b,Sc-Zha-Zh26,Sc-Zha-Zh27,Hu26} for an interplay between the (separating) Noether number and Davenport constant).
Simultaneously, in factorization theory, the monoid $\mathcal B (G)$ acts as a foundational model for understanding the algebraic and arithmetic properties of non-unique factorizations in transfer Krull domains (for instance, the rings of algebraic integers of algebraic number fields, the holomorphy rings in algebraic function fields, the maximal orders in central simple algebras, or the modules over hereditary noetherian prime(HNP) rings) and C-domains (for instance, non-principal orders in algebraic number fields, as well as in algebraic function fields, or non-local semilocal noetherian domains), providing a tangible combinatorial analogue of more complex algebraic structures (see \cite{Ge-Gr13,Gr13b,Oh19,Oh20,Oh-Zh20a,Ge-Gr-Oh-Zh22,Qu-Li-Tee22,An-Cz-Do-Sz25,Qu-Li-Tee25,Ge-Oh25} for a comprehensive study of combinatorial and arithmetic invariants of finite groups, as well as their interplay).
Formally, a sequence over $G$ means a finite collection of terms from $G$, where repetition is allowed and the order is disregarded.
In an algebraic context, a sequence can be viewed as an element of the free abelian monoid $\mathcal F (G)$ with basis $G$, endowed with the concatenation of sequences as the operation.
A product-one sequence is a sequence whose terms can be ordered such that their product in $G$ equals the identity element of $G$, and the monoid $\mathcal B (G)$ consists of those sequences.

Historically, the predominant focus of the literature has been on the abelian setting, where additive notation is typically used and product-one sequences are referred to as zero-sum sequences.
The combinatorial and arithmetical properties of zero-sum sequences have been extensively explored over the last several decades, revealing deep connections across algebra, combinatorics, cryptography, and number theory (see, for instance, \cite{Pl-Sc11,Ge-Zh20,Le-Oh23}).
The classical combinatorial invariants, such as the Davenport constant and the Erd\H{o}s-Ginzburg-Ziv constant, have been intensively studied since their initial work in the 1960s (see \cite{Sa-Za23} for strong recent progress on EGZ constant, and \cite{Gr13a} for a monograph on various combinatorial invariants).
Algebraically, the abelian case provides a remarkably elegant structure: if $G$ is abelian, then the monoid $\mathcal B (G)$ is Krull with class group isomorphic to $G$, and every class contains precisely one prime divisor.
Furthermore, for any abstract Krull monoid $H$ with class group $G$ such that every class contains prime divisors, there exists a transfer homomorphism from $H$ to $\mathcal B (G)$, which ensures that arithmetic structures of both $H$ and $\mathcal B (G)$ coincide.
The study of arithmetic of the monoid $\mathcal B (G)$ uses tools from additive combinatorics, which is a special branch of arithmetic combinatorics focussing only on addition and substraction.
The main goal is to describe arithmetical invariants precisely in terms of (classical) combinatorial invariants of $G$, such as the Davenport constant (see \cite{Ge-HK06,Ge09a,Sc16} for monographs and a survey on the arithmetic properties of Krull monoids).

However, translating from the abelian to non-abelian setting introduces profound structural changes and significant new challenges.
Recently, the investigation of combinatorial invariants of product-one sequences over non-abelian groups has been the subject of extensive study (see \cite{Oh-Zh20b,Zh21,Av-Br-Ri23,Ri25,Oh-Ri-Zha-Zh26}).
While $\mathcal B (G)$ itself remains a commutative monoid, as a submonoid of $\mathcal F (G)$, the arithmetic of its elements is deeply governed by the non-abelian nature of the base group $G$.
The most striking difference is the loss of the Krull property: if $G$ is non-abelian, then the monoid $\mathcal B (G)$ is no longer Krull.
Instead, it belongs to the broader and vastly more complex class of C-monoids.
C-monoids are submonoids of factorial monoids with finite class semigroup, and they include not only Krull domains with finite class group but also a large class of non-integrally closed Noetherian domains, as well as combinatorial monoids such as weighted zero-sum sequences and product-$K$ sequences (see, for instance, \cite{Fo-Ha06a,Ga-Ge09,Re13,Ge-Ra-Re15,Oh22,Bo-Me-Or-Sc22,Oh-Ya26} for the algebraic and arithmetic properties of C-monoids).
The monoid $\mathcal B (G)$ is a finitely generated C-monoid, which is a combinatorial analogue of more complicated C-monoids, and it opens the door to studying the algebraic structure of general C-monoids.
Indeed, it is seminormal if and only if the commutator subgroup of $G$ has order at most two if and only if its class semigroup is Clifford, i.e., a union of groups (see \cite{Oh19,Fa-Zh23}, and \cite{Ge-Zh19} for general C-monoids).
The concept of seminormality plays a fundamental role in various branches of algebra.
Originating in algebraic geometry to characterize varieties with mild singularities, such as normal crossings, it also guarantees the invariance of the Picard group under polynomial extensions in commutative algebra, which ensures foundational stability in algebraic K-theory as well (see \cite{Vi11} for a survey of seminormality for commutative rings and algebraic varieties).

As C-monoids satisfy the Mori property (the ascending chain condition on divisorial ideals), the monoid $\mathcal B (G)$ is atomic, i.e., every non-trivial product-one sequence over $G$ can be factored into atoms.
One of the most well-studied arithmetic invariants quantifying the non-uniqueness of these factorizations is the catenary degree.
Intuitively, it measures how difficult it is to connect different factorizations of the same element by a chain of nearby factorizations, and theoretically, it provides one measure of how far a monoid is from being factorial.
Formally, the catenary degree $\mathsf {c} (G)$ of the monoid $\mathcal B (G)$ is defined as the smallest integer $N$ with the following property: for each $S \in \mathcal B (G)$ and each two factorizations $z$ and $z'$ of $S$, there exist factorizations $z = z_0, z_1, \ldots, z_k = z'$ of $S$ such that, for each $i \in [1,k]$, $z_i$ arises from $z_{i-1}$ by replacing at most $N$ new atoms.
In the abelian setting, the catenary degree is well understood: all finite abelian groups with catenary degree at most 4 have been characterized, as well as all finite abelian groups whose catenary degree equals the (small) Davenport constant (see \cite{Ge-Gr-Sc11,Ge-Zh15}).
Despite decades of study, little more is known about the catenary degree, even for simple abelian cases such as $C_p \times C_p$ and $C^{r}_3$.
Moreover, it is known that both the sets of distances and catenary degrees for finite abelian groups are finite intervals (see \cite{Ge-Yu12,Ge-Zh19a}).
It is worthwhile to mention that these arithmetic invariants need not be intervals in general (see \cite{Ge-Sc17,Fa-Ge19} for realization results on these invariants).
In a more general setting, the monoid of product-$K$ sequences, where $K$ is a normal subgroup, was studied in \cite{Oh-Ya26}, and it is worth noting that its catenary degree can be determined by that of the monoid of product-one sequences over a suitable group.

In the present paper, we investigate the catenary degree of finite (not necessarily abelian) groups.
Extending the existing machinery to the non-abelian setting, we provide an explicitly characterization of all finite groups with catenary degree at most 3 (see Theorem~\ref{thm:class}).
Furthermore, we investigate an infinite class of finite groups whose monoids of product-one sequences are seminormal, and we show that their sets of distances and catenary degrees are finite intervals (see Theorem~\ref{thm:Int}).
Finally, using the methods from arithmetic combinatorics, we show that a specific non-abelian group belonging to this infinite class has catenary degree 4 (see Theorem~\ref{thm:Ca}).


\medskip
\section{Preliminaries} \label{2}
\medskip

Let $\mathbb N$ be the set of positive integers and $\mathbb N_0 = \mathbb N \cup \{ 0 \}$ be the non-negative integers.
For integers $a, b \in \mathbb Z$, we denote by $[a,b] = \{ x \in \mathbb Z \mid a \le x \le b \}$ the discrete interval.
For a subset $A \subseteq \mathbb Z$, the set $\Delta (A)$ of distances of $A$ is the set of all differences between consecutive elements of $A$, i.e., it consists of all $d \in \mathbb N$ for which there exists $\ell \in A$ such that $A \cap [\ell, \ell + d] = \{ \ell, \ell + d \}$.
In particular, $\Delta (A) = \emptyset$ if and only if $|A| \le 1$.

Let $G$ be a (multiplicatively written) finite group with identity element $1_G$.
For an element $g \in G$ and a subset $G_0 \subseteq G$, let $\ord (g) \in \mathbb N$ denote the order of $g$, and $\langle G_0 \rangle \leq G$ denote the subgroup generated by $G_0$.
We use the following standard notation of group theory:
\begin{itemize}
\item $[x, y] = xyx^{-1}y^{-1} \in G$ is the {\it commutator} of elements $x, y \in G$,

\smallskip
\item $G' = [G,G] = \langle [x,y] \mid x, y \in G \rangle$ is the {\it commutator subgroup} of $G$,

\smallskip
\item $Z (G) = \{ x \in G \mid gx = xg \text{ for all } g \in G \}$ is the {\it center} of $G$, and

\smallskip
\item $\Syl_p (G)$ is a {\it Sylow $p$-subgroup} of $G$ for a prime number $p$.
\end{itemize}
Furthermore, for a positive integer $n$, we denote by $C_n$ a {\it cyclic group} of order $n$, by $D_{2n}$ a {\it dihedral group} of order $2n$, and by $Q_{4n}$ a {\it dicyclic group} of order $4n$.
In particular, the quaternion group of order 8 is exactly the dicyclic group with $n = 2$.

\subsection*{Monoids}
Throughout this paper, a {\it monoid} means a commutative semigroup with identity element which satisfies the cancellation law.
Let $H$ be a multiplicative written monoid.
Then, $H^{\times}$ denotes the group of units in $H$, $\mathsf q (H)$ the quotient group of $H$, $H_{\red} = H/H^{\times}$ the associated reduced monoid, and $\mathcal A (H)$ the set of atoms (or irreducible elements) of $H$.
For any set $P$, we denote by $\mathcal F (P)$ the {\it free abelian monoid} with basis $P$ whose elements are of the form
\[
  a = \prod_{p \in P} p^{\mathsf v_p (a)} \in \mathcal F (P) \,,
\]
where $\mathsf v_p \colon H \to \mathbb N_0$ is the $p$-adic valuation of $a$.
In this case, $|a| = \sum_{p \in P} \mathsf v_p (a)$ is the {\it length} of $a$ and $\supp (a) = \{ p \in P \mid \mathsf v_p (a) \ge 1 \}$ is the {\it support} of $a$.

We denote by
\begin{itemize}
\item $H' = \{ x \in \mathsf q (H) \mid \text{there exists } N \in \mathbb N \text{ such that } x^{n} \in H \text{ for all } n \ge N \}$ the {\it seminormal closure} of $H$, and by

\smallskip
\item $\widehat{H} = \{ x \in \mathsf q (H) \mid \text{there exists } c \in H \text{ such that } cx^{n} \in H \text{ for all } n \in \mathbb N \}$ the {\it complete integral closure} of $H$.
\end{itemize}
Then, $H \subseteq H' \subseteq \widehat{H} \subseteq \mathsf q (H)$, and $H$ is called {\it seminormal} if $H = H'$ (equivalently, if $x \in \mathsf q (H)$ with $x^{2}, x^{3} \in H$, then $x \in H$), and called {\it completely integrally closed} if $H = \widehat{H}$.
The monoid $H$ is {\it Krull} if it is completely integrally closed and $v$-Noetherian (or Mori), i.e., it satisfies the ascending chain condition of divisorial ideals.
A {\it divisor theory} of $H$ is a homomorphism $\varphi \colon H \to F$ to a free abelian monoid $F := \mathcal F (P)$ satisfying the following two conditions:

\smallskip
\begin{itemize}
\item $\varphi$ is a {\it divisor homomorphism}, i.e., for all $a, b \in H$, $a \mid b$ in $H$ if and only if $\varphi (a) \mid \varphi (b)$ in $F$.

\smallskip
\item For each $p \in P$, there exist $a_1, \ldots, a_n \in H$ such that $p = \gcd \big( \varphi (a_1), \ldots, \varphi (a_n) \big)$.
\end{itemize}
Then, $H$ has a divisor theory if and only if $H$ is a Krull monoid (see \cite[Theorem 2.4.8]{Ge-HK06}).
An integral domain is a Krull domain if and only if its multiplicative monoid of non-zero elements is a Krull monoid.
A Noetherian domain is a Kull domain if and only if it is integrally closed.
Krull monoids is a well-understood class, which contains, for instance, the ring of algebraic integers in cumber fields and the holomorphy rings in algebraic function fields, with algebraic and arithmetic points of view (see \cite{Ge-Zh20} for a survey on the arithmetic of various monoids, such as Krull monoids and power monoids).

\subsection*{Arithmetic of monoids}
Let $H$ be a monoid.
The free abelian monoid $\mathsf Z (H) := \mathcal F \big( \mathcal A (H_{\red}) \big)$ with basis $\mathcal A (H_{\red})$ is called the {\it factorization monoid} of $H$, and the unique homomorphism $\pi \colon \mathsf Z (H) \to H_{\red}$, satisfying $\pi (u) = u$ for $u \in \mathcal A (H_{\red})$, is called the {\it factorization homomorphism} of $H$.
For $a \in H$, we denote by
\begin{itemize}
\item $\mathsf Z (a) = \pi^{-1} (aH^{\times}) \subseteq \mathsf Z (H)$ the {\it set of factorizations} of $a$, and by

\smallskip
\item $\mathsf L (a) = \{ |z| \mid z \in \mathsf Z (a) \}$ the {\it set of lengths} of $a$.
\end{itemize}
For each $a \in H$, $a \in H^{\times}$ if and only if $\mathsf L (a) = \{ 0 \}$.
In particular, $a \in \mathcal A (H)$ if and only if $\mathsf L (a) = \{ 1 \}$.
The monoid $H$ is said to be {\it atomic} if $\mathsf Z (a) \neq \emptyset$ for all $a \in H$, {\it factorial} if $|\mathsf Z (a)| = 1$ for all $a \in H$, and {\it half-factorial} if $|\mathsf L (a)| = 1$ for all $a \in H$.
Then, every factorial monoid is half-factorial.

Let $H$ be an atomic monoid.
We denote by $\mathcal L (H) = \{ \mathsf L (a) \mid a \in H \}$ the {\it system of sets of lengths} of $H$, and by
\[
  \Delta (H) = \bigcup_{\mathsf L \in \mathcal L (H)} \Delta (\mathsf L) \subseteq \mathbb N
\]
the {\it set of distances} of $H$, which is a crucial tool for investigating the structure of sets of lengths of $H$.
By the definition, $H$ is half-factorial if and only if $\Delta (H) = \emptyset$.
We define
\[
  \daleth^{*} (H) = \{ \min \big( \mathsf L \setminus \{ 2 \} \big) \mid \mathsf L \in \mathcal L (H) \text{ with } 2 \in \mathsf L \} \,,
\]
and $\daleth (H) = \sup \daleth^{*} (H)$ with the convention that $\min \emptyset = \sup \emptyset = 0$.
It follows by the definition that $\daleth (H) = 0$ if and only if $\mathsf L (u \cdot v) = \{ 2 \}$ for all $u, v \in \mathcal A (H)$, and that $\daleth (H) \le 2 + \sup \Delta (H)$.
Two factorizations $z, z' \in \mathsf Z (H)$ can be written in the form
\[
  z = u_1 \cdot \ldots \cdot u_{\ell} \cdot v_1 \cdot \ldots \cdot v_m \quad \mbox{ and } \quad z' = u_1 \cdot \ldots \cdot u_{\ell} \cdot w_1 \cdot \ldots \cdot w_n
\]
with
\[
  \{ v_1, \ldots, v_m \} \cap \{ w_1, \ldots, w_n \} = \emptyset \,,
\]
where $\ell, m, n \in \mathbb N_0$ and $u_1, \ldots, u_{\ell}, v_1, \ldots, v_m, w_1, \ldots, w_n \in \mathcal A (H_{\red})$.
Then, $\gcd (z, z') = u_1 \cdot \ldots \cdot u_{\ell}$, and we call $\mathsf d (z, z') = \max \{ m, n \} \in \mathbb N_0$ the {\it distance} between $z$ and $z'$.

Let $a \in H$ and $N \in \mathbb N_0 \cup \{ \infty \}$.
For any factorizations $z, z' \in \mathsf Z (a)$ of $a$, a finite sequence $z_0, \ldots, z_k$ in $\mathsf Z (a)$ is said to be an {\it N-chain of factorizations} from $z$ and $z'$ if $z = z_0$, $z' = z_k$ and $\mathsf d (z_{i-1}, z_i) \le N$ for every $i \in [1,k]$.
In this case, we say that $z$ and $z'$ can be {\it concatenated} by an $N$-chain.
We define $\mathsf c (a)$ to be the smallest integer $N \in \mathbb N_0 \cup \{ \infty \}$ such that any two factorizations $z, z' \in \mathsf Z (a)$ can be concatenated by an $N$-chain.
By definition, $\mathsf c (a) = 0$ if and only if $|\mathsf Z (a)| = 1$, and if $|\mathsf Z (a)| \ge 2$, then $2 \le \mathsf c (a) \le \sup \mathsf L (a)$ and $2 + \sup \Delta \big( \mathsf L (a) \big) \le \mathsf c (a)$.
Then, the set
\[
  \mathsf {Ca} (H) = \{ \mathsf c (a) \mid a \in H \mbox{ with } \mathsf c (a) \ge 2 \} \subseteq \mathbb N_0
\]
is the {\it set of catenary degrees} of $H$, and $\mathsf c (H) = \sup \mathsf {Ca} (H) \in \mathbb N_0 \cup \{ \infty \}$ is the {\it catenary degree} of $H$.
Then, $\mathsf c (H) = 0$ if and only if $H$ is factorial, and if $H$ is not factorial, then $2 + \sup \Delta (H) \le \mathsf c (H)$.
Therefore, we obtain that
\begin{equation} \label{eq:ine}~
  \daleth (H) \le 2 + \sup \Delta (H) \le \mathsf c (H)
\end{equation}
(see \cite[Lemma 3.1]{Ge-Gr-Sc11} for more details).

\subsection*{Product-one sequences}
By a {\it sequence} over $G$, we mean a finite collection of terms from $G$, where repetition is allowed and the order of terms are disregarded.
In algebraic context, we consider sequences as elements of the free abelian monoid $\mathcal F (G)$ with basis $G$, endowed with the concatenation of sequences as the operation denoted by $\bdot$.
In order to avoid confusion between the operations in $G$ and $\mathcal F (G)$, we denote the operation in $G$ without an explicit symbol and use brackets for all exponentiation in $\mathcal F (G)$.
Thus, a sequence $S \in \mathcal F (G)$ is of the form
\begin{equation} \label{eq:sequence}~
  S = g_1 \bdot \ldots \bdot g_{\ell} = {\small \prod}^{\bullet}_{g \in G} g^{[\mathsf v_g (S)]} \,,
\end{equation}
where $g_1, \ldots, g_{\ell} \in G$, $\mathsf v_g (S) = |\{ i \in [1,\ell] \mid g_i = g \}|$ is the {\it multiplicity} of $g$ in $S$, and $|S| = \ell = \sum_{g \in G} \mathsf v_g (S)$ is the {\it length} of $S$.
For a sequence $S$, we denote by $\supp (S) = \{ g \in G \mid \mathsf v_g (S) \ge 1 \}$ the {\it support} of $S$.
For sequences $S, T \in \mathcal F (G)$, $T$ is called a {\it divisor} of $S$ in $\mathcal F (G)$ if $S = T \bdot W$ for some $W \in \mathcal F (G)$, and denoted by $T \mid S$.
In this case, $T$ is said to be a {\it subsequence} of $S$, and
\[
  S \bdot T^{[-1]} := {\small \prod}^{\bullet}_{g \in G} g^{[\mathsf v_g (S) - \mathsf v_g (T)]}
\]
denote the subsequence of $S$ obtained by removing the terms of $T$ from $S$.
More generally, if $T^{[n]} \mid S$ for some $n \in \mathbb N$, then $S \bdot T^{[-n]} := S \bdot \big( T^{[n]} \big)^{[-1]} \in \mathcal F (G)$.
For a subset $H \subseteq G$, we denote by $S_H$ the subsequence of $S$ consisting of all terms from $H$.
On the other hand, we denote by $S^{-1} : = g^{-1}_1 \bdot \ldots \bdot g^{-1}_{\ell}$.

Let $S \in \mathcal F (G)$ be a sequence as in (\ref{eq:sequence}).
Then, the {\it set of products} of $S$ is denoted by
\[
  \pi (S) = \{ g_{\sigma (1)} \cdots g_{\sigma (\ell)} \in G \mid \sigma \mbox{ is a permutation on } [1,\ell] \} \,,
\]
and since $G/G'$ is an abelian group, it is easy to see that $\pi (S)$ is contained in a $G'$-coset.
For the convention, $|S| = 0$ if and only if $S$ is a trivial sequence, and we use $\pi (S) = \{ 1_G \}$.
A generic element of $\pi (S)$ will be denoted by $\sigma (S)$.
In particular, $\pi (S) = \{ \sigma (S) \}$ if $G$ is abelian.
The sequence $S$ is called
\begin{itemize}
\item {\it product-one} if $1_G \in \pi (S)$,

\smallskip
\item {\it product-one free} if $S$ has no non-trivial product-one subsequences,
\end{itemize}
and any ordered product that equals $1_G$ in $\pi (S)$ is called a {\it product-one equation} of $S$.
Then, the set
\[
  \mathcal B (G) = \{ S \in \mathcal F (G) \mid 1_G \in \pi (S) \}
\]
of all product-one sequences over $G$ forms a submonoid of $\mathcal F (G)$, and is called the {\it monoid of product-one sequences} over $G$.
A {\it minimal product-one sequence} is an atom in the monoid $\mathcal B (G)$, i.e., it is a product-one sequence but cannot be factored into two non-trivial product-one subsequences, and so $\mathcal A (G) := \mathcal A (\mathcal B (G))$ is the set of all minimal product-one sequences.

We denote by
\begin{itemize}
\item $\mathsf D (G) = \sup \{ |S| \mid S \in \mathcal A (G) \}$ the {\it large Davenport constant} of $G$,

\smallskip
\item $\mathsf d (G) = \sup \{ |S| \mid S \text{ is a product-one free sequence} \}$ the {\it small Davenport constant} of $G$.
\end{itemize}
Then, we obtain that $\mathsf d (G) + 1 \le \mathsf D (G) \le |G|$ with equality in the first bound if $G$ is abelian, and equality in the second bound if and only if $G$ is isomorphic to $C_m$ or $D_{2n}$ with $n$ odd (see \cite[Lemma 2.4]{Ge-Gr13} and \cite[Corollary 3.12]{Cz-Do-Ge16}).
It is worth noting that we are not aware of any example of a finite non-abelian group $G$ with $\mathsf d (G) + 1 = \mathsf D (G)$.
We refer the reader to \cite{Ge-Gr13,Gr13b,Qu-Li-Tee22,Qu-Li-Tee25,Oh26} for recent works on the Davenport constants of finite groups.

The monoid $\mathcal B (G)$ is a (combinatorially defined) finitely generated C-monoid (see \cite[Theorem 3.2]{Cz-Do-Ge16}), which is a non-completely integrally closed analogue of Krull monoids.
C-monoids are a large class of Mori monoids and they contain not only Krull monoids with finite class group but also non-integrally closed Noetherian domains (see \cite{Re13,Ge-Zh19,Oh22} for algebraic properties of general C-monoids and their class semigroups).
Another examples of combinatorial C-monoids are monoids of weighted zero-sum sequences and product-$K$ sequences, with a normal subgroup $K$ (see \cite{Bo-Me-Or-Sc22,Oh-Ya26} for their detailed studies for the algebraic and arithmetic properties).
We refer the reader to \cite{Cz-Do-Ge16,Oh19,Oh20,Ge-Gr-Oh-Zh22,Fa-Zh23,Ge-Oh25} for studies on the algebraic and arithmetic properties of the monoid of product-one sequences for (in)finite groups.
We summarize here the algebraic structure for the monoid $\mathcal B (G)$ (although the relevant results are studied across several aforementioned papers, they can all be found in the recent paper \cite{Oh-Ya26} by taking $K$ to be trivial).

\smallskip
\begin{lemma} $($\cite{Oh-Ya26}$)$ \label{lem:structure}~
Let $G$ be a finite group.
\begin{enumerate}
\item $\mathcal B (G)$ is a (half-)factorial monoid if and only if $|G| \le 2$.

\smallskip
\item The following statements are equivalent:
         \begin{enumerate}
         \smallskip
         \item $G$ is abelian.

         \smallskip
         \item $\mathcal B (G)$ is a Krull monoid.

         \smallskip
         \item $\mathcal B (G)$ is a weakly Krull monoid.

         \smallskip
         \item $\mathcal B (G)$ is a transfer Krull monoid.

         \smallskip
         \item The inclusion $\mathcal B (G) \hookrightarrow \mathcal F (G)$ is a divisor homomorphism.
         \end{enumerate}

\smallskip
\item $\mathcal B (G)$ is a seminormal monoid if and only if $| G' | \le 2$.
\end{enumerate}
\end{lemma}

\noindent
Thus, if $G$ is non-abelian, then the monoid $\mathcal B (G)$ is a non-Krull C-monoid, so that the inclusion $\mathcal B (G) \hookrightarrow \mathcal F (G)$ is not a divisor homomorphism.
For instance, for the quaternion group $Q_8$, both $S = J^{[4]} \bdot I^{[2]}$ and $T = J^{[4]}$ are product-one sequences, even they are minimal (see Theorem~\ref{thm:Davenport5}), with $T \mid S$ in $\mathcal F (Q_8)$, but $T \nmid S$ in $\mathcal B (Q_8)$.
This is a striking difference for study of arithmetic properties of the monoid of product-one sequences between abelian and non-abelian settings.

As usual, we follows the convention of writing $* (G)$ instead of $* \big( \mathcal B (G) \big)$ for all arithmetical concepts $* (H)$ defined for a monoid $H$.
For instance,
\[
  \Delta (G) := \Delta \big( \mathcal B (G) \big)\,, \,\, \daleth (G) := \daleth \big( \mathcal B (G) \big) \,, \,\, \mbox{ and } \,\, \mathsf c (G) := \mathsf c \big( \mathcal B (G) \big) \,.
\]
We recall known results on the catenary degree for abelian groups, and we refer the reader to \cite{Ge-Gr-Sc11,Ge-Zh19a} for further details.

\smallskip
\begin{lemma}  $($\cite{Ge-Gr-Sc11}$)$ \label{lem:abel}~
Let $G$ be a finite abelian group with $|G| \ge 3$, say $G \cong C_{n_1} \times \cdots \times C_{n_r}$ with $1 \lneq n_1 \mid \cdots \mid n_r$.
\begin{enumerate}
\item $\big\{ n_r, 1 + \sum_{i=1}^{r}  \lfloor \frac{n_i}{2} \rfloor \big\} \le \daleth (G) \le \mathsf c (G) \le \mathsf D (G)$.

\smallskip
\item $\mathsf c (G) = 3$ if and only if $G$ is isomorphic to one of the groups $C_3$, $C_2 \times C_2$, or $C_3 \times C_3$.

\smallskip
\item $\mathsf c (G) = 4$ if and only if $G$ is isomorphic to $C_4$, $C_2 \times C_4$, $C_2 \times C_2 \times C_2$, or $C_3 \times C_3 \times C_3$.
\end{enumerate}
\end{lemma}

\smallskip
It is worth noting that $\mathsf c (G) \le \mathsf D (G)$ is a canonical bound in the abelian setting (see \cite{Ge-Zh19a} for characterizations of finite groups whose catenary degree equals the (small) Davenport constant); however, it remains unknown whether this inequality holds in the non-abelian setting.
Despite decades of study, the exact value of the catenary degree remains largely unknown, even for simple abelian cases such as $C_p \times C_p$ for a prime $p$ and $C^{r}_3$ for $r \ge 4$ (see \cite[Appendix B, Problem 10]{Ge-Gr-Zh26}).


\medskip
\section{The catenary degree for non-abelian groups} \label{3}
\medskip

In this section, we focus on the catenary degree for non-abelian groups.
We begin by establishing a lower bound on $\daleth (G)$ and $\mathsf c (G)$ for groups containing non-commuting elements.

\smallskip
\begin{lemma} \label{lem:min}~
Let $G$ be a finite group and $g \in G$ with $\ord (g) = 3$.
If there exists an element $h \in G$ with $\ord (h) \in [2,3]$ such that $gh \neq hg$, then $\daleth (G) \ge 4$.
\end{lemma}

\begin{proof}
Suppose that $G$ has an element $h$ with $\ord (h) = 3$ (resp., $\ord (h) = 2$) such that
\begin{equation} \label{eq:noncommut}
  gh \neq hg \,.
\end{equation}
Then, $ghg \neq 1_G$, for otherwise $h = g$, a contradiction to (\ref{eq:noncommut}) (resp., $2 = \ord (h) = \ord (g) = 3$).
Let
\[
  U = g \bdot h \bdot g \bdot (g^{-1}h^{-1}g^{-1}) \quad \big(\mbox{resp., } \,\, U = g \bdot h \bdot g \bdot (g^{-1}hg^{-1}) \big)
\]
and
\[
  V = U^{-1} = g^{-1} \bdot h^{-1} \bdot g^{-1} \bdot (ghg) \quad \big(\mbox{resp., } \,\, V = U^{-1} = g^{-1} \bdot h \bdot g^{-1} \bdot (ghg) \big)
\]
be minimal product-one sequences of length 4.
Then, $\mathsf L (U \bdot V) \supseteq \{ 2, 4 \}$, and it suffices to show that $3 \notin \mathsf L (U \bdot V)$; indeed, $4 = \min \big( \mathsf L ( U \bdot V) \setminus \{ 2 \} \big) \le \daleth (G)$, whence the assertion follows.
Assume to the contrary that $3 \in \mathsf L (U \bdot V)$, so that $U \bdot V = W_1 \bdot W_2 \bdot W_3$ for some $W_1, W_2, W_3 \in \mathcal A (G)$ with $|W_1| \ge |W_2| \ge |W_3|$.
By the construction of $U$ and $V$, we infer that $|W_1| = |W_2| = 3$, $|W_3| = 2$.
Let $W_1 = g_1 \bdot g_2 \bdot g_3$.
If $g_i = g^{-1}_j$ for some $i \neq j$, then we may assume that $W_1 = g_1 \bdot g^{-1}_1 \bdot g_3$.
So, a straightforward computation using the product-one equation shows that $g_3 = 1_G$, a contradiction.
Thus, it follows that $g_i \neq g^{-1}_j$ for all $i \neq j$, so that $W_2 = W^{-1}_1$.

Suppose now that $ghg \mid W_1$.
Then, since $|W_3| = 2$, it follows that either $W_3 = h \bdot h^{-1}$ (resp., $W_3 = h \bdot h$) or $W_3 = g \bdot g^{-1}$.
If $W_3 = h \bdot h^{-1}$ (resp., $W_3 = h \bdot h$), then either $W_1 = g \bdot g \bdot (ghg)$ or $W_1 = g^{-1} \bdot g^{-1} \bdot (ghg)$.
In the former case, the product-one equation shows that either $h = 1_G$ or $g = h^{-1}$, a contradiction to (\ref{eq:noncommut}).
In the latter case, we obtain that $h = 1_G$, which is also a contradiction.
Hence, we assume that $W_3 = g \bdot g^{-1}$.
Then, $W_1$ is one of the following forms;
\[
  g \bdot h \bdot (ghg) \,, \quad g \bdot h^{-1} \bdot (ghg) \,, \quad g^{-1} \bdot h \bdot (ghg) \,, \quad g^{-1} \bdot h^{-1} \bdot (ghg)
\]
(resp., we have that either $W_1 = g \bdot h \bdot (ghg)$ or $W_1 = g^{-1} \bdot h \bdot (ghg)$).

If $W_1 = g \bdot h \bdot (ghg)$, then by the product-one equation, we must have that $hg = gh^{-1}$ (resp., $gh = hg$, a contradiction to (\ref{eq:noncommut})), equivalently $gh = h^{-1}g$.
Since $\ord (g) = 3$, it follows that
\[
  hg^{-1} = (hg)g = (gh^{-1})g = g(h^{-1}g) = g(gh) = g^{-1}h \,,
\]
equivalently $gh = hg$, a contradiction to (\ref{eq:noncommut}).

If $W_1 = g \bdot h^{-1} \bdot (ghg)$, then it follows by the product-one equation that
\[
  h = g(ghg) = g^{-1}hg \quad \mbox{ or } \quad h = (ghg)g = ghg^{-1} \,.
\]
In either case, we obtain that $gh = hg$, a contradiction to (\ref{eq:noncommut}).

If $W_1 = g^{-1} \bdot h \bdot (ghg)$, then it follows by the product-one equation that
\[
  h^{-1} = g^{-1}(ghg) = hg \quad \mbox{ or } \quad h^{-1} = (ghg)g^{-1} = gh \,
\]
(resp., $h = hg$ or $h = gh$, implying that $g = 1_G$, a contradiction).
Since $\ord (h) = 3$, we obtain, in either case, that $h = h^{-2} = g$, a contradiction to (\ref{eq:noncommut}).

If $W_1 = g^{-1} \bdot h^{-1} \bdot (ghg)$, then it follows by the product-one equation that
\[
  h = g^{-1} (ghg) = hg \quad \mbox{ or } \quad h = (ghg)g^{-1} = gh \,.
\]
In either case, we obtain that $g = 1_G$, a contradiction.

Suppose that $(ghg)^{-1} \mid W_1$.
Since $W_2 = W^{-1}_1$, the same argument as used in the former case leads to a contradiction by swapping the roles of $W_1$ and $W_2$.

Suppose that $ghg \nmid W_1$ and $(ghg)^{-1} \nmid W_1$.
Then, we must have that $W_3 = (ghg) \bdot (ghg)^{-1}$.
Since $W_2 = W^{-1}_1$, it suffices to consider the case where either $W_1 = g \bdot h \bdot g$ or $W_1 = g \bdot h^{-1} \bdot g$ (resp., $W_1 = g \bdot h \bdot g$).
In the former case, the product-one equation shows that $h = g$, and in the latter case, we obtain that $h = g^{-1}$.
In either case, this leads to a contradiction to (\ref{eq:noncommut}) (resp., $2 = \ord (h) = \ord (g) = 3$).
\end{proof}

\smallskip
\begin{proposition} \label{pro:non-abel}~
If $G$ is a finite non-abelian group, then $\mathsf c (G) \ge 4$.
\end{proposition}

\begin{proof}
If $G$ has an element $g$ with $\ord (g) \ge 4$, then $U = g^{[\ord (g)]}$ and $V = (g^{-1})^{[\ord (g)]}$ are minimal product-one sequences.
It follows that $\mathsf L (U \bdot V) = \{ 2, \ord (g) \}$, whence $4 \le \daleth (G) \le \mathsf c (G)$.
Now, we assume that
\begin{equation} \label{eq:order}~
  \ord (g) \in [2,3] \,\, \mbox{ for all } \, g \in G \setminus \{ 1_G \} \,, \mbox{ and }
\end{equation}
\[
  |G| = 2^{n_1} 3^{n_2} \quad \mbox{ for some } \,\, n_1, n_2 \in \mathbb N_{0}, \mbox{ not all zero.}
\]
Since $G$ is non-abelian, we may suppose that $n_2 \ge 1$.
By Cauchy's Theorem, there exists an element $g \in G$ with $\ord (g) = 3$.
Assume that $n_1 = 0$, i.e., $|G| = 3^{n_2}$.
Since $G$ is non-abelian, (\ref{eq:order}) ensures that there exist $h \in G$ with $\ord (h) = 3$ and $gh \neq hg$.
Thus, it follows by Lemma~\ref{lem:min} that $4 \le \daleth (G) \le \mathsf c (G)$.
Hence, we must have that $|G| = 2^{n_1} 3^{n_2}$ for $n_1, n_2 \in \mathbb N$.
Again by Cauchy's Theorem, there exists an element $f \in G$ with $\ord (f) = 2$.
Then, it follows by (\ref{eq:order}) that $gf \neq fg$, for otherwise we obtain that $\ord (gf) = 6$, a contradiction to (\ref{eq:order}).
Hence, Lemma~\ref{lem:min} implies that $4 \le \daleth (G) \le \mathsf c (G)$.
\end{proof}

\smallskip
We are now ready to classify all finite groups with catenary degree at most 3 as follows.

\smallskip
\begin{theorem} \label{thm:class}~
Let $G$ be a finite group.
\begin{enumerate}
\item $\mathsf c (G) \le 2$ if and only if $|G| \le 2$.

\smallskip
\item The following statements are equivalent:
      \begin{enumerate}
      \smallskip
      \item[(a)] $\mathsf c (G) = 3$.

      \smallskip
      \item[(b)] $\Delta (G) = \{ 1 \}$.

      \smallskip
      \item [(c)] $\daleth (G) = 3$.

      \smallskip
      \item[(d)] $G$ is isomorphic to $C_3$, or $C_2 \times C_2$, or $C_3 \times C_3$.
      \end{enumerate}
\end{enumerate}
\end{theorem}

\begin{proof}
1. This part is already known and included only for completeness.
If $|G| \le 2$, then it follows by Lemma~\ref{lem:structure}.1 that $\mathcal B (G)$ is a factorial monoid, so that $\mathsf c (G) = 0$.
If $\mathsf c (G) \le 2$, then in view of (\ref{eq:ine}), we infer that $\Delta (G) = \emptyset$, so that the monoid $\mathcal B (G)$ is half-factorial.
Thus, it follows again by Lemma~\ref{lem:structure}.1 that $|G| \le 2$.

\smallskip
2. (a) $\Rightarrow$ (b) By item 1, $|G| \ge 3$, and so $\mathcal B (G)$ is not factorial by Lemma~\ref{lem:structure}.1.
In view of (\ref{eq:ine}), we obtain that $2 + \max \Delta (G) \le \mathsf c (G) = 3$, whence $\max \Delta (G) \le 1$.

(b) $\Rightarrow$ (c) Then, $\mathcal B (G)$ is not factorial, and so $|G| \ge 3$ by Lemma~\ref{lem:structure}.1.
Thus, in view of (\ref{eq:ine}), $\daleth (G) \le 2 + \max \Delta (G) = 3$, and hence it suffices to show that $3 \le \daleth (G)$.
If $g \in G$ with $\ord (g) \ge 3$, then we are done.
Suppose that $G$ has no elements of order 3.
This means that $G$ must be abelian, and since $\Delta (G) \neq \emptyset$, we infer that $G$ has distinct two elements of order 2, say $g$ and $h$.
Then, it is east to see that $g \bdot h \bdot (gh)$, $g^{[2]}$, $h^{[2]}$, and $(gh)^{[2]}$ are all minimal product-one sequences, and $\big( g \bdot h \bdot (gh) \big)^{[2]} = g^{[2]} \bdot h^{[2]} \bdot (gh)^{[2]}$, whence $3 \le \daleth (G)$.

(c) $\Rightarrow$ (d) Since $\daleth (G) = 3$, we infer that every element of $G$ has the order at most 3.
To show that $G$ must be abelian, assume to the contrary that $G$ is non-abelian.
Then, $G$ must have at least one $g \in G$ with $\ord (g) = 3$.
Since $G$ is non-abelian, there exists an element $h \in G$ such that $gh \neq hg$.
In either case when $\ord (h) = 2$ or $\ord (h) = 3$, Lemma~\ref{lem:min} shows that there exist minimal product-one sequences $U, V$ over $G$ with $\{ 2, 4 \} \subseteq \mathsf L (U \bdot V)$ and $3 \notin \mathsf L (U \bdot V)$.
Hence, $4 \le \daleth (G)$, a contradiction.
Thus, $G$ must be abelian, and it follows by Lemma~\ref{lem:abel}.1 that $G$ is isomorphic to $C_3$, or $C_2 \times C_2$, or $C_3 \times C_3$.

(c) $\Rightarrow$ (a) This follows by Lemma~\ref{lem:abel}.2.
\end{proof}

\smallskip
We now consider the following splitting property:

\medskip
\begin{itemize}
\item[\namedlabel{itm:P}{\bf P}.] If $g_1 \bdot \ldots \bdot g_{\ell} \in \mathcal A (G)$, and $g_1 = h_1 h_2$ with $h_1, h_2 \in G$, then the sequence $h_1 \bdot h_2 \bdot g_2 \bdot \ldots \bdot g_{\ell} \in \mathcal B (G)$ is either an atom or a product of only two atoms.
\end{itemize}

\smallskip
\noindent
Note that every abelian group satisfies Property {\bf P}.
We have so far only two types in non-abelian case as follows: It has been shown that the quaternion group satisfies Property {\bf P} using the structure of minimal product-one sequences with the help of a computer program; however, a dihedral group $D_{2n}$ with $n \ge 9$ does not satisfy Property {\bf P} (see \cite[Example 5.3]{Oh20}).

We say that a finite group is {\it minimal non-abelian} if it is non-abelian and every proper subgroup is abelian.
R\'edei \cite{Red47} gave a complete characterization of these groups, and in particular, for a prime $p$, minimal non-abelian $p$-groups can be characterized (see \cite[Exercise 1.8a]{Be08} for written in a modern language).
We present a characterization of a specific subclass of finite groups whose commutator subgroup has two elements (see \cite{Mi38} for a general information for groups whose commutator subgroup has two element).

\smallskip
\begin{proposition} \label{pro:min}~
Let $G$ be a finite group.
The following statements are equivalent:
\begin{enumerate}
\item[(a)] $|G'| = 2$ and $\Syl_2 (G)$ is a direct product of an abelian group and a minimal non-abelian group.

\smallskip
\item[(b)] $G = G_1 \times G_2$, where $G_1$ is abelian and $G_2$ is isomorphic to one of the following groups.
		\begin{itemize}
		\smallskip
		\item[(i)] A metacyclic group $\langle a, b \mid a^{2^{m}} = b^{2^{n}} = 1_G, \text{ and } \, ba = a^{1+2^{m-1}}b \rangle$ of order $2^{m+n}$ for $m \ge 2$, and $n \ge 1$.
		
		\smallskip
		\item[(ii)] A non-metacyclic group $\langle a, b \mid a^{2^{m}} = b^{2^{n}} = c^{2} = 1_G, \text{ and } \, ab = cba, ac = ca, bc = cb \rangle$ of order $2^{m+n+1}$ for $m + n \gneq 2$.
		
		\smallskip
		\item[(iii)] The quaternion group $Q_8$ of order 8.
		\end{itemize}
\end{enumerate}
\end{proposition}

\begin{proof}
(a) $\Rightarrow$ (b) Since $|G'| = 2$, it is easy to show that $G' \subseteq Z (G)$ and $g \in Z (G)$ for all $g \in G$ with odd order.
Since $G' \subseteq Z (G)$, it follows that $G$ is nilpotent, and so $G$ is a direct product of its Sylow subgroups.
Since every element of odd order is central, every Sylow $p$-subgroup, with $p \ge 3$, is abelian.
Thus, we infer that
\begin{equation} \label{eq:cs2}~
  G = \Pi_{p} \Syl_p (G) = G_0 \times \Syl_2 (G)
\end{equation}
for some abelian group $G_0$.
By assumption, $\Syl_2 (G) = H \times G_2$ for some abelian group $H$ and a minimal non-abelian group $G_2$.
Thus, $G = G_0 \times \Syl_2 (G) = (G_0 \times H) \times G_2 = G_1 \times G_2$, where $G_1$ is an abelian group.
Since $G_2$ is minimal non-abelian 2-group, the assertion follows by \cite[Exercise 1.8a]{Be08}.

(b) $\Rightarrow$ (a) Since $G' = (G_1 \times G_2)' = G'_1 \times G'_2 \cong G'_2$, it is clear that $|G'| = 2$, and so $G$ is nilpotent.
If $|G_1|$ is odd, then $\Syl_2 (G) = G_2$.
If $|G_1|$ is even, then $G_1 = \Syl_2 (G_1) \times G_0$ for some abelian group $G_0$, and thus $\Syl_2 (G) = \Syl_2 (G_1) \times G_2$, with $\Syl_2 (G_1)$ abelian.
Thus, it suffices to show that three groups described in (b) is minimal non-abelian.
It is well-known that $Q_8$ is minimal non-abelian.

Let $K$ be a group described in (ii).
Since $K$ is finite, it is enough to show that every maximal subgroup of $K$ is abelian.
Let $M$ be a maximal subgroup of $K$.
Since $K$ is 2-group, it follows that $M$ is normal in $K$ and $K/M \cong C_2$, which ensures that $K' \subseteq M$ and $k^{2} \in M$ for every $k \in K$.
In particular, $a^{2}, b^{2}, c \in M$, and so $Z := \langle a^{2}, b^{2}, c \rangle \le M$.
We observe that $Z \subseteq Z (K)$, so that $Z$ is an abelian subgroup of $K$.
Since $K = \langle a, b \rangle$, we infer that $K/Z = \langle aZ, bZ \rangle \cong C_2 \times C_2$, and $M/Z$ is a proper subgroup of $K/Z$, i.e., $M/Z$ is a cyclic group.
To show that $M$ is abelian, let $m_1, m_2 \in M$.
Then, $m_1 = m^{i}z_1$ and $m_2 = m^{j}z_2$ for some $i, j \in [1,2]$ and $z_1, z_2 \in Z$.
Since $Z \subseteq Z(K)$, we obtain that $m_1 m_2 = m^{i} z_1 m^{j} z_2 = m^{i} m^{j} z_1 z_2 = m^{j} z_2 m^{i} z_1 = m_2 m_1$, and thus $M$ is abelian.

Now let $K$ be the group described in (i), and $M$ be a maximal subgroup.
Then, it is easy to see that $a^{2}, b^{2} \in M$, and $Z := \langle a^{2}, b^{2} \rangle \subseteq Z (K)$.
Moreover, $c:= a^{2^{m-1}} \in K' \setminus \{ 1_K \}$, and since $c = (a^{2})^{2^{m-2}}$, we obtain that $c \in M$.
Thus, we infer that $M$ is abelian by the same argument as used in the case (ii).
\end{proof}

\smallskip
In particular, by (\ref{eq:cs2}), the minimal non-abelian groups whose commutator subgroup has two elements are precisely the 2-groups listed in Proposition~\ref{pro:min}.
Although we are not aware of all finite groups satisfying Property {\bf P}, the following result shows that at least the infinite class of finite groups described in Proposition~\ref{pro:min} does satisfy Property {\bf P}.

\smallskip
\begin{theorem} \label{thm:P}~
Let $G$ be a finite group.
If $G$ is abelian, or if $|G'| = 2$ and the Sylow $2$-subgroup is a direct product of an abelian group and a minimal non-abelian group, then $G$ satisfies Property {\bf P}.
\end{theorem}

\begin{proof}
Suppose that either $G$ is abelian, or that $|G'| = 2$ and the $\Syl_2 (G)$ is a direct product of an abelian group and a minimal non-abelian group.
To show that $G$ satisfies Property {\bf P}, let $U \in \mathcal A (G)$ and $g \in \supp (U)$ with $g = h_1 h_2$ for some $h_1, h_2 \in G$.
Then, $U' := h_1 \bdot h_2 \bdot U \bdot g^{[-1]} \in \mathcal B (G)$, and so
\[
  U' = V_1 \bdot \ldots \bdot V_k
\]
for some atoms $V_1, \ldots, V_k \in \mathcal A (G)$.
Assume to the contrary that $k \ge 3$.
If $h_1$ and $h_2$ split into distinct two atoms $V_i$, say $V_1$ and $V_2$, then the product-one equations ensure that
\[
  V_1 = V'_1 \bdot h_1 \quad \mbox{ and } \quad V_2 = h_2 \bdot V'_2
\]
with $\sigma (V'_1) h_1 = 1_G = h_2 \sigma (V'_2)$ for some $\sigma (V'_1) \in \pi (V'_1)$ and $\sigma (V'_2) \in \pi (V'_2)$.
Then, $V''_1 := V'_1 \bdot (h_1h_2) \bdot V'_2 = V'_1 \bdot g \bdot V'_2$ is a product-one sequence, and thus $U = V''_1 \bdot V_3 \bdot \ldots \bdot V_k$.
Since $k \ge 3$, it follows that $U$ can be factored into a product of at least two product-one sequences, contradicting that $U \in \mathcal A (G)$.
Thus, there must exist some $i \in [1,k]$ such that $V_i$ contains both $h_1$ and $h_2$, say $V_1$.
Since $k \ge 3$, we infer that $V_2$ and $V_3$ are proper product-one subsequences of $U$.
If $G$ is abelian, then $U \bdot V^{[-1]}_2$ must be again a product-one sequence, which leads to a contradiction to $U \in \mathcal A (G)$.
Thus, from now on, we may assume that $G$ is non-abelian.
Then, since $U \in \mathcal A (G)$, we obtain that $U \bdot V^{[-1]}_2 \notin \mathcal B (G)$.
Since $U$ and $V_2$ are product-one sequences, we obtain that $\pi (U \bdot V^{[-1]}_2) \subseteq G'$.
Since $|G'| = 2$, it follows that $|\pi \big( U \bdot V^{[-1]}_2 \big) | = |G' \setminus \{ 1_G \}| = 1$, and thus all terms in $U \bdot V^{[-1]}_2$ commute with each other.
By swapping the role between $V_2$ and $V_3$, the same argument shows that all terms in $U \bdot V^{[-1]}_3$ also commute with each other.
If every elements from $V_2$ and $V_3$ commute with each other, then $U$ can be viewed as a sequence over abelian subgroup of $G$, and since $U$ has a proper product-one subsequence $V_2$, we infer that $U \bdot V^{[-1]}_2$ is also a product-one sequence, a contradiction.

Thus, there exist $g_1 \in \supp (V_2)$ and $g_2 \in \supp (V_3)$ such that $g_1 g_2 \neq g_2 g_1$.
By Proposition~\ref{pro:min}, $G = G_1 \times G_2$ with abelian $G_1$ and minimal non-abelian $G_2$.
Let $\varphi \colon G \to G_2$ be a projection.
Then, $\varphi (g_1)$ and $\varphi (g_2)$ are non-commuting elements in $G_2$, and since $G_2$ is minimal non-abelian, it follows that $\big\langle \varphi (g_1), \varphi (g_2) \big\rangle = G_2$.
Since every term in $\varphi \big( U \bdot (V_2 \bdot V_3)^{[-1]} \big)$ commutes with both $\varphi (g_1)$ and $\varphi (g_2)$, we infer that all elements in $\supp \big( \varphi \big( U \bdot (V_2 \bdot V_3)^{[-1]} \big) \big)$ belong to the center of $G_2$, in particular, all terms in $\varphi \big( V_1 \bdot g \bdot (h_1 \bdot h_2)^{[-1]} \big)$ are center elements in $G_2$.
This means that $h_1$ and $h_2$ commute with all terms in $V_1$, and so the product-one equation of $V_1$ shows that
\[
  1_G = (h_1 h_2) \sigma \big( V_1 \bdot (h_1 \bdot h_2)^{[-1]} \big) \in \pi \big( V_1 \bdot g \bdot (h_1 \bdot h_2)^{[-1]} \big) \,.
\]
Thus, $U = \big( V_1 \bdot g \bdot (h_1 \bdot h_2)^{[-1]} \big) \bdot V_2 \bdot V_3 \bdot \ldots \bdot V_k$ is a product of at least 3 product-one sequences, again a contradiction.
Therefore, $G$ does satisfy Property {\bf P}.
\end{proof}

\smallskip
In the previous theorem, the condition of minimality with respect to being non-abelian is necessary for Property {\bf P}, as the following example shows.

\smallskip
\begin{example} \label{ex:P}~
\begin{enumerate}
\item Let $G = \langle a, b, c, d \mid a^{4} = b^{2} = d^{2} = 1_G, c^{2} = a^{2}, ba = a^{-1}b, ac = ca, ad = da, bc = cb, bd = db, dc = a^{2}cd \rangle$ be a finite group of order 32 (we take its generators and relations from the database of finite groups \cite{GNs}, and this group is identified with $\texttt{SmallGroup(32,49)}$, which represents the $49^{\textnormal{th}}$ group of order 32 in the Small Groups Library of GAP \cite{GAP}).

	 Then, $G' = \{ 1_G, a^{2} \} = Z (G)$, and both $\langle a, b \rangle$ and $\langle c, d \rangle$ are isomorphic to $D_8$.
	 Consider the sequence
	 \[
	   S = b^{[2]} \bdot (ab)^{[2]} \bdot d \bdot (a^{2}d) \in \mathcal F (G) \,.
	 \]
	 \begin{itemize}
	 \item[(a)] In the group $G$, $a^{2}$ is a central element implies that $b (ab) b (ab) d (a^{2}d) = b a^{2}b d a^{2}d = 1_G$, whence $S \in \mathcal B (G)$.
	
	 \smallskip
	 \item[(b)] To show $S \in \mathcal A (G)$, we assume to the contrary that $S = S_1 \bdot S_2$ for non-trivial $S_1, S_2 \in \mathcal B (G)$.
	 		By symmetry, we further assume that $|S_1| \le |S_2|$.
			If $|S_1| = 2$ and $|S_2| = 4$, then since $a, b, c, d$ are generators of $G$, we obtain that $S_1 \in \{ b^{[2]}, (ab)^{[2]} \}$.
			In either case, $d \bdot (a^{2}d) \mid S_2$, and since $d$ commutes with both $a$ and $b$, we infer that $\pi (S_2) = \{ a^{2} \}$, a contradiction to $S_2 \in \mathcal B (G)$.
			If $|S_1| = 3 = |S_2|$, then since $d$ commutes with both $a$ and $b$ and $\ord (d) = 2$, it follows that $d \bdot (a^{2}d) \mid S_i$ for some $i \in [1,2]$, say $S_1$.
			Since $|S_1| = 3$, we obtain that $S_1 = b \bdot d \bdot (a^{2}d)$ or $S_1 = (ab) \bdot d \bdot (a^{2}d)$, but it follows that $\pi (S_1) = \{ a^{2}b \}$ or $\pi (S_1) = \{ a^{3}b \}$, which, in either case, yields a contradiction.
			Therefore, $S$ must be an atom in $\mathcal B (G)$.
			
	\smallskip
	\item[(c)] Let $a^{2}d = (cd) (a^{2}c)$ for $cd, a^{2}c \in G$.
		      Then, $(a^{2}c) \bdot d \bdot (cd) \in \mathcal A (G)$, because $1_G = (a^{2}c) (cd) d$ and it never be factored into a product of two product-one subsequences.
		      Thus, a sequence
		      \[
		        S' := \big( b^{[2]} \big) \bdot \big( (ab)^{[2]} \big) \bdot \big( (a^{2}c) \bdot d \bdot (cd) \big)
		      \]
		      is a product of 3 atoms.
	\end{itemize}
	
	\smallskip
	Therefore, although $|G'| = 2$, $G$ does not satisfy Property {\bf P}.
	
\smallskip
\item Let $D_{2n}$ be the dihedral group of order $2n$ with $n \ge 3$ odd.
	 Then, a sequence $U = \tau^{[n]} \bdot (\alpha^{i}\tau)^{[n]}$ for $i \in [1,n-1]$ with $\gcd (i, n) = 1$, is a minimal product-one sequence of length $2n$ for some generators $\alpha, \tau \in D_{2n}$ (see \cite[Theorem 4.1]{Oh-Zh20a}).
	 Since $n \ge 3$ is odd, there exists $j \in [1,n-1]$ such that $2j \equiv -i \pmod{n}$.
	 Then, $\tau = \alpha^{j} (\alpha^{n-j}\tau)$, and so
	 \[
	   U' = \alpha^{j} \bdot \alpha^{n-j}\tau \bdot \tau^{[n-1]} \bdot (\alpha^{i}\tau)^{[n]} = \big( \alpha^{n-j}\tau \bdot \alpha^{j} \bdot \alpha^{i}\tau \big) \bdot \big( \tau \bdot \tau \big)^{[\frac{n-1}{2}]} \bdot \big( \alpha^{i}\tau \bdot \alpha^{i}\tau \big)^{[\frac{n-1}{2}]}
	 \]
	 is a product of at least 3 atoms.
	 Thus, $D_{2n}$, with $n$ odd, does not satisfy Property {\bf P}.
	 More generally, the dihedral group $D_{2n}$, with $n \ge 9$, does not satisfy Property {\bf P} (see \cite[Example 5.3]{Oh20}).
\end{enumerate}
\end{example}

\smallskip
It is worthwhile to mention that the set $\Delta (G)$ of distances is a finite interval for finite groups satisfying Property {\bf P} (see \cite[Theorem 5.5]{Oh20}).
Next, we show that $\daleth^{*} (G)$ and the set $\mathsf {Ca} (G)$ of catenary degrees are also finite intervals for finite groups satisfying Property {\bf P}.
We refer the reader to \cite[Proposition 3.3]{Fa-Ge19} and \cite[Lemma 4.6]{Ge-Zh19a} for the results in the abelian setting.
Some arithmetic properties carry over along the lines of the proofs as in the abelian setting, but we need the following observation.

\smallskip
\begin{lemma} $($\cite[Lemma 5.1]{Oh20}$)$ \label{lem:fac}~
Let $k, \ell \in \N$ with $k \lneq \ell$, and $U_1, \ldots, U_k, V_1, \ldots, V_{\ell} \in \mathcal A (G)$ such that $U_1 \bdot \ldots \bdot U_k = V_1 \bdot \ldots \bdot V_{\ell}$.
Then, there exist $\mu \in [1,k], \,\, \lambda, \lambda' \in [1,\ell]$ with $\lambda \neq \lambda',$ and $g_1, g_2 \in G$ such that $U_{\mu} = g_1 \bdot g_2 \bdot \ldots \bdot g_m$ with $m \ge 2$, $1_G = g_1 \ldots g_m$, $g_1 \t V_{\lambda}$, and $g_2 \t V_{\lambda^{'}}$ in $\mathcal F (G)$.
\end{lemma}

\smallskip
\begin{lemma} \label{lem:Ca}~
Let $G$ be a finite group with $|G| \ge 3$.
Suppose that $G$ satisfies Property {\bf P}.
\begin{enumerate}
\item For every $S \in \mathcal B (G)$ with $2 \in \mathsf L (S)$ and $\min \big(\mathsf L (S) \setminus \{ 2 \}\big) \ge 4$, there exists $T \in \mathcal B (G)$ with $2 \in \mathsf L (T)$ such that $|T| \lneq |S|$ and $\min \big( \mathsf L (T) \setminus \{ 2 \}\big) \ge \min \big(\mathsf L (S) \setminus \{ 2 \}\big) - 1$.

\smallskip
\item For every $S \in \mathcal B (G)$ with $\mathsf c (S) \ge 4$, there exists $T \in \mathcal B (G)$ such that $|T| \lneq |S|$ and $\mathsf c (T) \ge \mathsf c (S) - 1$.
\end{enumerate}
\end{lemma}

\begin{proof}
1. Let $S \in \mathcal B (G)$ with $\ell = \min \big( \mathsf L (S) \setminus \{ 2 \} \big) \ge 4$.
There exist $U_1, U_2, V_1, \ldots, V_{\ell} \in \mathcal A (G)$ such that
\[
  S = U_1 \bdot U_2 = V_1 \bdot \ldots \bdot V_{\ell} \,,
\]
where $U_1 \bdot U_2$ has no factorization of length in $[3, \ell-1]$.
Since $1_G \in \mathcal B (G)$ is a prime element, we may assume that $1_G$ does not pop up in any factorization of $S$.
This allows us to suppose that $U_1, U_2, V_1, \ldots, V_{\ell}$ are of length at least 2.
By Lemma~\ref{lem:fac}, after renumbering if necessary, there exist $g_1, g_2 \in G$ such that
\[
  g_1 \bdot g_2 \mid U_1 \,, \quad g_1 \mid V_1 \,, \quad \mbox{ and } \quad g_2 \mid V_2 \,,
\]
where $U_1 = g_1 \bdot g_2 \bdot U$ with $g_1 g_2 \sigma (U) = 1_G$ for some $\sigma (U) \in \pi (U)$.
Let $g_0 = g_1 g_2 \in G$, so that $U'_1 := g_0 \bdot U \in \mathcal A (G)$.
Since $V_1$ and $V_2$ are product-one sequences, we can write
\[
  V_1 = V'_1 \bdot g_1 \quad \mbox{ and } \quad V_2 = g_2 \bdot V'_2
\]
with $\sigma (V'_1) g_1 = 1_G = g_2 \sigma (V'_2)$ for some $\sigma (V'_1) \in \pi (V'_1)$ and $\sigma (V'_2) \in \pi (V'_2)$.
Then, $V'_1 := V'_1 \bdot V'_2 \bdot g_0 \in \mathcal B (G)$, and so
\[
  T := U'_1 \bdot U_2 = V'_1 \bdot V_3 \bdot \ldots \bdot V_{\ell} \in \mathcal B (G) \quad \mbox{ with } \,\, |T| \lneq |S| \,.
\]
Let $T = W_1 \bdot \ldots \bdot W_t$ with $t \in \mathbb N$, $W_1, \ldots, W_t \in \mathcal A (G)$ for all $i \in [1,t]$, and $g_0 \mid W_1$.
If we set $W_1 = g_0 \bdot W'_1$, then since $G$ satisfies Property {\bf P}, it follows that $W := g_1 \bdot g_2 \bdot W'_1$ is a product of at most two atoms, implying that $S$ has a factorization of length $t$ or $t+1$.
We infer that $T$ has no factorization of length in $[3, \ell-2]$, and hence
\[
  \min \big( \mathsf L (T) \setminus \{ 2 \}\big) \ge \ell-1 = \min \big( \mathsf L (S) \setminus \{ 2 \} \big) - 1 \,.
\]
Thus, $T$ is the desired sequence.

\smallskip
2. Let $S \in \mathcal B (G)$ with $\mathsf c (S) \ge 4$, and assume that $S$ has minimal length among all product-one sequences having the same catenary degree as $S$.
Then, there exist $z, z' \in \mathsf Z (S)$ with $k = |z| \le |z'| = \ell$ such that $\mathsf d (z,z') = \mathsf c (S)$ and there is no $(\mathsf c (S)-1)$-chain concatenating $z$ and $z'$.
Then, $\ell \ge \mathsf c (S)$, and since $1_G \in \mathcal B (G)$ is a prime element, we may assume that $1_G$ doe not divide $S$.
Thus, we can write
\[
  z = U_1 \bdot \ldots \bdot U_k \in \mathsf Z (S) \quad \mbox{ and } \quad z' = V_1 \bdot \ldots \bdot V_{\ell} \in \mathsf Z (S) \,,
\]
where $U_i, V_j \in \mathcal A (G)$ with $|U_i| \ge 2$ and $|V_j| \ge 2$ for all $i \in [1,k]$ and $j \in [1,\ell]$.
By Lemma~\ref{lem:fac}, we may suppose that
\[
  U_1 = g_1 \bdot g_2 \bdot U \quad \mbox{ with } \, g_1 \mid V_1 \mbox{ and } g_2 \mid V_2 \,,
\]
where $U = g_3 \bdot \ldots \bdot g_{|U_1|}$ and $g_1 g_2 g_3 \cdots g_{|U_1|} = 1_G$.
Let $g_0 = g_1 g_2 \in G$, so that $U'_1 := g_0 \bdot U \in \mathcal A (G)$.
Since $V_1$ and $V_2$ are product-one sequences, we can write
\[
  V_1 = V'_1 \bdot g_1 \quad \mbox{ and } \quad V_2 = g_2 \bdot V'_2
\]
with $\sigma (V'_1) g_1 = 1_G = g_2 \sigma (V'_1)$ for some $\sigma (V'_1) \in \pi (V'_1)$ and $\sigma (V'_2) \in \pi (V'_2)$, and then we obtain a product-one sequence
\[
  V' := V'_1 \bdot V'_2 \bdot g_0 \in \mathcal B (G) \,.
\]
Then, $T = (g_1 \bdot g_2)^{[-1]} \bdot S \bdot g_0 \in \mathcal B (G)$ with $|T| \lneq |S|$.
If $V' = W_1 \bdot \ldots \bdot W_t$, where $t \in \mathbb N$, $W_i \in \mathcal A (G)$ for all $i \in [1,t]$, and $g_0 \mid W_1$, then $T$ has two factorizations
\[
  \overline{z} = U'_1 \bdot U_2 \bdot \ldots \bdot U_k \in \mathsf Z (T) \quad \mbox{ and } \quad \overline{z}' = W_1 \bdot \ldots \bdot W_t \bdot V_3 \bdot \ldots \bdot V_{\ell} \in \mathsf Z (T) \,.
\]

Assume to the contrary that there exists a $(\mathsf c (S)-2)$-chain $\overline{z} = \overline{z}_1, \ldots, \overline{z}_n = \overline{z}'$ in $\mathsf Z (T)$.
For $i \in [1,n]$, write
\[
  \overline{z}_i = Z_{i,1} \bdot \ldots \bdot Z_{i, m_i} \in \mathsf Z (T) \,,
\]
where $m_i \in \mathbb N$, $Z_{i,j} \in \mathcal A (G)$ for $j \in [1,m_i]$, and $g_0 \mid Z_{i,1}$.
For each $i \in [1,n]$, we obtain a product-one sequence
\[
  Z'_{i,1} := (Z_{i,1} \bdot g^{[-1]}_0) \bdot g_1 \bdot g_2 \in \mathcal B (G) \,,
\]
and we take a factorization $\overline{y}_{i,1} \in \mathsf Z (Z'_{i,1})$.
Since $G$ satisfies Property {\bf P}, it follows that $\overline{y}_{i,1}$ is a product of at most 2 atoms of $\mathcal B (G)$.
Then, for each $i \in [1,n]$, we have a factorization
\[
  z_{2i} := \overline{y}_{i,1} \bdot Z_{i,2} \bdot \ldots \bdot Z_{i,m_i} \in \mathsf Z (S) \,.
\]
Now, we take $i \in [1,n-1]$, and we put $z_{2i+1}$ as follows:
\begin{itemize}
\item[(a)] $z_{2i+1} := \overline{y}_{i,1} \bdot \overline{z}_{i+1} \bdot Z^{[-1]}_{i+1,\zeta} \in \mathsf Z (S)$ if $Z_{i+1,\zeta} = Z_{i,1}$ for some $\zeta \in [1, m_{i+1}]$,

\smallskip
\item[(b)] $z_{2i+1} := \overline{y}_{i+1,1} \bdot \overline{z}_i \bdot Z^{[-1]}_{i,\zeta} \in \mathsf Z (S)$ if $Z_{i, \zeta} = Z_{i+1,1}$ for some $\zeta \in [1,m_i]$, or

\smallskip
\item[(c)] $z_{2i+1} := z_{2i}$ if $Z_{i+1,\zeta} \neq Z_{i,1}$ for all $\zeta \in [1,m_{i+1}]$ and $Z_{i,\xi} \neq Z_{i+1,1}$ for all $\xi \in [1,m_i]$.
\end{itemize}
According to the construction, if (a) occurs, then $\mathsf d (z_{2i}, z_{2i+1}) = \mathsf d (\overline{z}_i, \overline{z}_{i+1}) \le \mathsf c (S) - 2 \le \mathsf c (S) - 1$ and $\mathsf d (z_{2i+1}, z_{2i+2}) \le 1 + \max \big\{ |\overline{y}_{i,1}|, |\overline{y}_{i+1,1}| \big\} \le 3 \le \mathsf c (S) - 1$.

If (b) occurs, then $\mathsf d (z_{2i}, z_{2i+1}) \le 3 \le \mathsf c (S) - 1$ and $\mathsf d (z_{2i+1}, z_{2i+2}) = \mathsf d (\overline{z}_i , \overline{z}_{i+1}) \le \mathsf c (S) - 1$.

If (c) occurs, then $\mathsf d (z_{2i}, z_{2i+2}) \le 2 + \mathsf d \left( \overline{z}_i \bdot Z^{[-1]}_{i,1}, \, \overline{z}_{i+1} \bdot Z^{[-1]}_{i+1,1} \right) \le 2 + (\mathsf c (S) - 3 ) = \mathsf c (S) - 1$.

Since $G$ satisfies Property {\bf P}, it follows that $\mathsf d (z, z_2) \le 2 \le \mathsf c (S) - 1$, which ensures that there exists a $(\mathsf c (S) - 1)$- chain of factorizations of $S$ from $z$ to
\[
  z_{2n} := \overline{y}_{n, 1} \bdot W_2 \bdot \ldots \bdot W_t \bdot V_3 \bdot \ldots \bdot V_{\ell} \in \mathsf Z (S) \,,
\]
whence $V_1 \bdot V_2 = \overline{y}_{n,1} \bdot W_2 \bdot \ldots \bdot W_t$.
Since $\ell \ge \mathsf c (S) \ge 4$, it follows that $| V_1 \bdot V_2 | \lneq |S|$, so that $\mathsf c (V_1 \bdot V_2) \le \mathsf c (S) - 1$ by the minimality of $|S|$.
Hence, there exists a $(\mathsf c (S) - 1)$-chain of factorizations of $S$ concatenating $z_{2n}$ and $z'$, and therefore there must exists a $(\mathsf c (S) - 1)$-chain of factorizations of $S$ concatenating $z$ and $z'$, a contradiction.
It follows that there is no $(\mathsf c (S) - 2)$-chain of factorizations concatenating $\overline{z}$ and $\overline{z}'$, which ensures that $\mathsf c (T) \ge \mathsf c (S) - 1$.
Thus, $T$ is the desired sequence.
\end{proof}

\smallskip
The following theorem shows that Property {\bf P} guarantees very structured arithmetical invariants for the monoid $\mathcal B (G)$.
As its consequence, we conclude that all groups discussed in Theorem~\ref{thm:P} possess well-behaved algebraic and arithmetic structures.
However, these arithmetic invariants need not be intervals in general: indeed, for any finite subset $A \subseteq \mathbb N$, there is a Krull monoid $H$ with $\Delta (H) = A$ (see \cite[Theorem 1.1]{Ge-Sc17}) and with $\mathsf {Ca} (H) = A \setminus \{ 1 \}$ and $\daleth^{*} (H) = A \setminus \{ 1, 2 \}$ (see \cite[Proposition 3.2]{Fa-Ge19}).
We also refer the reader to \cite{On-Pe18} for a realization result on numerical monoids.

\smallskip
\begin{theorem} \label{thm:Int}~
Let $G$ be a finite group with $|G| \ge 3$.
\begin{enumerate}
\item If $G$ satisfies Property {\bf P}, then
	\begin{enumerate}
	\smallskip
	\item[(i)] $\Delta (G)$ is a finite interval with $\min \Delta (G) = 1$,
	
	\smallskip
	\item[(ii)] $\daleth^{*} (G)$ is a finite interval with $\min \daleth^{*} (G) = 3$, and
	
	\smallskip
	\item[(iii)] $\mathsf {Ca} (G)$ is a finite interval with $\min \mathsf {Ca} (G) = \left\{ \begin{array}{ll}
	                                                        3 & \hbox{if \, $\mathsf D (G) = 3$} \\
	                                                        2 & \hbox{if \, $\mathsf D (G) \ge 4$} \,.
	                                                        \end{array} \right.$
	\end{enumerate}
	
\medskip
\item If $G$ is either abelian, or that $|G'| = 2$ and the Sylow-$2$ subgroup is minimal non-abelian, then the monoid $\mathcal B (G)$ is seminormal whose $\Delta (G)$, $\daleth^{*} (G)$, and $\mathsf {Ca} (G)$ are all finite intervals.
\end{enumerate}
\end{theorem}

\begin{proof}
1. Suppose that $G$ satisfies Property~{\bf P}.

\smallskip
(i) See \cite[Theorem 5.5]{Oh20}.

\smallskip
(ii) Since $\daleth (G)$ is finite, we take $S_0 \in \mathcal B (G)$ with minimal length and $2 \in \mathsf L (S_0)$ such that $\daleth (G) = \min \big( \mathsf L (S_0) \setminus \{ 2 \} \big)$.
By Lemma~\ref{lem:Ca}.1, there exists $S_1 \in \mathcal B (G)$ with minimal length and $2 \in \mathsf L (S_0)$ such that $|S_1| \lneq |S_0|$ and $\min \big( \mathsf L (S_1) \setminus \{ 2 \} \big) \ge \min \big( \mathsf L (S_0) \setminus \{ 2 \} \big) - 1$.
By the minimality of $S_0$, $|S_1| \lneq |S_0|$ implies that $\min \big( \mathsf L (S_1) \setminus \{ 2 \} \big) \lneq \daleth (G) = \min \big( \mathsf L (S_0) \setminus \{ 2 \} \big)$, and thus
\[
  \daleth (G) - 1 = \min \big( \mathsf L (S_0) \setminus \{ 2 \} \big) - 1 = \min \big( \mathsf L (S_1) \setminus \{ 2 \} \big) \in \daleth^{*} (G) \,.
\]
Again by Lemma~\ref{lem:Ca}.1, there exists $S_2 \in \mathcal B (G)$ with minimal length and $2 \in \mathsf L (S_2)$ such that $|S_2| \lneq |S_1|$ and $\min \big( \mathsf L (S_2) \setminus \{ 2 \} \big) \ge \min \big( \mathsf L (S_1) \setminus \{ 2 \} \big) - 1$.
Then, it follows by the minimality of $|S_1|$ that $\min \big( \mathsf L (S_2) \setminus \{ 2 \} \big) \lneq \min \big( \mathsf L (S_1) \setminus \{ 2 \} \big)$, so that
\[
  \daleth (G) - 2 = \min \big( \mathsf L (S_0) \setminus \{ 2 \} \big) - 2 = \min \big( \mathsf L (S_1) \setminus \{ 2 \} \big) - 1 = \min \big( \mathsf L (S_2) \setminus \{ 2 \} \big) \in \daleth^{*} (G) \,.
\]
Continuing this process, there exists $S_n \in \mathcal B (G)$ with $2 \in \mathsf L (S_n)$ such that
\[
  \min \daleth^{*} (G) = \daleth (G) - n = \min \big( \mathsf L (S_n) \setminus \{ 2 \} \big) \in \daleth^{*} (G) \,.
\]
Now, if there exists $g \in G$ with $\ord (g) \ge 3$, then $3 = \min \big( \mathsf L \big( (g \bdot g^{-1})^{[\ord (g)]} \big) \setminus \{ 2 \} \big) \in \daleth^{*} (G)$.
Suppose that every non-identity element in $G$ has order 2.
It follows that $G$ must be abelian, and if $g_1, g_2 \in G \setminus \{ 1_G \}$ are distinct two elements, then $3 = \min \big( \mathsf L \big( (g_1g_2)^{[2]} \bdot g^{[2]}_1 \bdot g^{[2]}_2 \big) \setminus \{ 2 \} \big) \in \daleth^{*} (G)$.
Thus, we infer that $\min \daleth^{*} (G) = 3$.

\smallskip
(iii) Since the assertion has been done in \cite[Lemma 4.6]{Ge-Zh19a} for abelian groups, we here assume that $G$ is a non-abelian group satisfying Property {\bf P}.
Then, by Proposition~\ref{pro:non-abel}, $\mathsf c (G) \ge 4$.
Let $S_0 \in \mathcal B (G)$ with minimal length such that $\mathsf c (S_0) = \mathsf c (G)$.
Then, by Lemma~\ref{lem:Ca}.2, there exists $S_1 \in \mathcal B (G)$ with minimal length such that $|S_1| \lneq |S_0|$ and $\mathsf c (S_1) \ge \mathsf c (S_0) - 1$.
By the minimality of $S_0$, $|S_1| \lneq  |S_0|$ implies that $\mathsf c (S_1) \lneq \mathsf c (G)$, and hence $\mathsf c (G) - 1 = \mathsf c (S_1) \in \mathsf {Ca} (G)$.

Applying Lemma~\ref{lem:Ca}.2 again, we infer that there exists $S_2 \in \mathcal B (G)$ with minimal length such that $\mathsf c (S_2) \ge \mathsf c (S_1) - 1$.
By the minimality of $S_1$, $|S_2| \lneq |S_1|$ implies that $\mathsf c (S_2) \lneq \mathsf c (S_1) = \mathsf c (G) - 1$.
Thus, $\mathsf c (G) - 2 = \mathsf c (S_1) - 1 = \mathsf c (S_2) \in \mathsf {Ca} (G)$.
Repeating this process, there must exist $S_n \in \mathcal B (G)$ such that
\[
  \min \mathsf {Ca} (G) = \mathsf c (G) - n = \mathsf c (S_n) \in \mathsf {Ca} (G) \,.
\]
Hence, $\mathsf {Ca} (G) = [ \min \mathsf {Ca} (G), \mathsf c (G)]$ is an interval.
Now, we consider the minimum of $\mathsf {Ca} (G)$.
Since $G$ is non-abelian, it follows by \cite[Lemma 4.6]{Oh19} or \cite[Theorem 1.1]{Oh26} that $\mathsf D (G) \ge 6$, and so we only need to show that $2 \in \mathsf {Ca} (G)$.

If $G$ has an element $g$ with $\ord (g) = n \ge 4$, then $2 = \mathsf c \big(g^{[n]} \bdot (g^{n-2} \bdot g^{2}) \big) \in \mathsf {Ca} (G)$.
Thus, we may suppose that
\begin{equation} \label{eq:23}~
  \ord (g) \in [2,3] \quad \mbox{ for all } \,\, g \in G \setminus \{ 1_G \} \,.
\end{equation}
Since $G$ is non-abelian, there exists at least one $g \in G$ with $\ord (g) = 3$.
If $G$ has an element $h$ with $\ord (h) = 2$, then we must have that $gh \neq hg$, and so $G$ has a subgroup $\langle g, h \rangle \cong D_6$.
It follows by \cite[Theorem 5.1]{Ge-Gr-Oh-Zh22} that $\mathsf {Ca} (D_6) = [2,3]$, so that $2 \in \mathsf {Ca} (G)$.
Assume that there exists $h \in G \setminus \langle g \rangle$ with $\ord (h) = 3$.
Then, $\ord (gh) \in [2,3]$ by (\ref{eq:23}).
If either $gh = hg$, or that $gh \neq hg$ and $\ord (gh) = 3$, then $(gh) \bdot g^{[2]} \bdot h^{[2]} \in \mathcal A (G)$ and $2 = \mathsf c \big( (gh)^{[3]} \bdot ((gh) \bdot g^{[2]} \bdot h^{[2]}) \big) \in \mathsf {Ca} (G)$.
If $gh \neq hg$ and $\ord (gh) = 2$, then we have that $(gh)g \neq g(gh)$.
Thus, $D_6 \cong \langle g, gh \rangle$ is a subgroup of $G$, and hence $2 \in \mathsf {Ca} (G)$, as in the previous case when $\ord (h) = 2$.

\smallskip
2. The assertion follows by Lemma~\ref{lem:structure}.3, Theorem~\ref{thm:P}, and Item 1.
\end{proof}

\smallskip
Although a finite group $G$ does not satisfy Property {\bf P}, and the monoid $\mathcal B (G)$ is not seminormal, there exists a class of finite groups for which the conclusions of Theorem~\ref{thm:Int} still hold, as follows.

\smallskip
\begin{remark} \label{rmk:dihe}~
Let $G = D_{2n}$ be the dihedral group of order $2n$ with $n \ge 3$ odd.
By Example~\ref{ex:P}.3, $G$ does not satisfy Property {\bf P}.
However, $\Delta (G) = [1, 2n-2]$ and $\mathsf {Ca} (G) = [2,2n]$ are finite intervals (see \cite[Theorem 5.1]{Ge-Gr-Oh-Zh22}).
Moreover, along to the same line, one can show that $\daleth^{*} (G) = [3,2n]$ is also a finite interval with $\mathsf L (U \bdot U_k) = \{ 2, 2n-k \}$ for $k \in [0,n]$, where $\alpha, \tau \in G$ are generators, and $U = \tau^{[n]} \bdot (\alpha\tau)^{[n]}$ and $U_k = \alpha^{[k]} \bdot \tau^{[n-k]} \bdot (\alpha\tau)^{[n-k]}$ are minimal product-one sequences.
\end{remark}


\medskip
\section{The quaternion group} \label{4}
\medskip

In this section, we fully describe the sets of distances and catenary degrees for the quaternion group $Q_8$.
By Proposition~\ref{pro:min}, this group belongs to the class of minimal non-abelian groups whose commutator subgroup has two elements.
Moreover, by Theorem~\ref{thm:Int}, this group possess well-behaved algebraic and arithmetic structures.
It is worthwhile to mention that, for every finite abelian group $G$, $\mathsf c (G) \le \mathsf D (G) = \boldsymbol{\beta} (G)$ is well-known result (see \cite[Proposition 6.1.3]{Ge-Gr-Zh26} for the first bound, and \cite{Sc91} for the second bound).
We have that $\boldsymbol{\beta_{\textnormal{sep}}} (G) \le \boldsymbol{\beta} (G)$ by the definition, and $\boldsymbol{\beta_{\textnormal{sep}}} (Q_8) = \boldsymbol{\beta} (Q_8) = 6$ (see \cite{Cz-Do-Sz18,Do-Sc25}).
Then, Theorem~\ref{thm:Ca} shows that $\mathsf c (Q_8) = 4 \le 6 = \boldsymbol{\beta_{\textnormal{sep}}} (Q_8) = \boldsymbol{\beta} (Q_8)$.
This naturally leads to the question of whether $\mathsf c (G) \le \boldsymbol{\beta} (G)$ (or $\boldsymbol{\beta_{\textnormal{sep}}} (G))$ for every finite group $G$.

Now let $Q_8 = \{ \pm{E}, \pm{I}, \pm{J}, \pm{K} \}$ be the quaternion group of order 8 with the relation $I^{2} = J^{2} = K^{2} = IJK = -E$ and the identity $1_{Q_8} = E$.
Note that any two elements $x, y \in Q_8 \setminus \{ 1_{Q_8}, -E \}$ with $x \neq \pm{y}$ generate $Q_8$, so that $Q_8 = \langle x, y \mid x^{4} = 1_{Q_8}, y^{2} = x^{2}, \text{ and } yx = x^{-1}y \rangle$.
Then, it follows by \cite[Theorem 1.1]{Ge-Gr13} that
\[
  \mathsf d (Q_8) = 4 \quad \mbox{ and } \quad \mathsf D (Q_8) = 6 \,,
\]
and we start with a characterization of minimal product-one sequences of all lengths except 2, as follows.

\smallskip
\begin{theorem} \label{thm:Davenport5}~
Let $S \in \mathcal F (Q_8)$ be a sequence.
\begin{enumerate}
\item $S$ is a minimal product-one sequence of length $|S| = 6$ if and only if there exist $\alpha, \tau \in Q_8$ such that $Q_8 = \langle \alpha, \tau \mid \alpha^{4} = 1_{Q_8}, \tau^{2} = \alpha^{2}, \mbox{ and } \tau\alpha = \alpha^{-1}\tau \rangle$ and
\begin{equation} \label{eq:Da6}~
  S = \alpha^{[4]} \bdot \tau^{[2]} \,.
\end{equation}

\smallskip
\item $S$ is a minimal product-one sequence of length $|S| = 5$ if and only if there exist $\alpha, \tau \in Q_8$ such that $Q_8 = \langle \alpha, \tau \mid \alpha^{4} = 1_{Q_8}, \tau^{2} = \alpha^{2}, \mbox{ and } \tau\alpha = \alpha^{-1}\tau \rangle$ and $S$ has one of the following forms:
\begin{equation} \label{eq:Da5}~
  \alpha^{[3]} \bdot \tau \bdot \alpha\tau \,, \quad \alpha^{[3]} \bdot \tau \bdot \alpha^{3}\tau \,, \quad \alpha^{2} \bdot \alpha^{[2]} \bdot \tau^{[2]} \,, \quad \mbox{ or } \quad \alpha^{2} \bdot \alpha \bdot \alpha^{-1} \bdot \tau \bdot \tau^{-1} \,.
\end{equation}

\smallskip
\item $S$ is a minimal product-one sequence of length $|S| = 4$ if and only if there exist $\alpha, \tau \in Q_8$ such that $Q_8 = \langle \alpha, \tau \mid \alpha^{4} = 1_{Q_8}, \tau^{2} = \alpha^{2}, \text{ and } \tau \alpha = \alpha^{-1} \tau \rangle$ and $S$ has one of the following forms:
\begin{equation} \label{eq:Da4}~
  \alpha^{[4]} \,, \quad \alpha^{[2]} \bdot \tau^{[2]} \,, \quad \alpha^{[2]} \bdot \tau \bdot \tau^{-1} \,, \quad \alpha^{2} \bdot \alpha \bdot \tau \bdot \alpha\tau \,, \quad \mbox{ or } \quad \alpha^{2} \bdot \alpha \bdot \tau \bdot \alpha^{3}\tau \,.
\end{equation}

\smallskip
\item $S$ is a minimal product-one sequence of length $|S| = 3$ if and only if there exist $\alpha, \tau \in Q_8$ such that $Q_8 = \langle \alpha, \tau \mid \alpha^{4} = 1_{Q_8}, \tau^{2} = \alpha^{2}, \text{ and } \tau \alpha = \alpha^{-1} \tau \rangle$ and $S$ has one of the following forms:
\begin{equation} \label{eq:Da3}~
  \alpha \bdot \tau \bdot \alpha\tau \,, \quad \alpha \bdot \tau \bdot \alpha^{3}\tau \,, \quad \mbox{ or } \quad \alpha^{2} \bdot \alpha^{[2]} \,.
\end{equation}
\end{enumerate}
\end{theorem}

\begin{proof}
We fix generators $\alpha, \tau \in Q_8$ such that $Q_8 = \langle \alpha, \tau \mid \alpha^{4} = 1_{Q_8}, \tau^{2} = \alpha^{2}, \text{ and } \alpha \tau = \alpha^{-1}\tau \rangle$.
Then, we can write $Q_8 = \langle \alpha \rangle \cup G_0$, where $G_0 = \{ \alpha^{i}\tau \mid i \in [0,3] \}$.

\smallskip
1. See \cite[Theorem 4.3]{Oh-Zh20a}.

\smallskip
2. Let $S \in \mathcal A (Q_8)$ with $|S| = 5$. Since $\langle \alpha \rangle$ is a cyclic group of order 4, $|S| = 5 \gneq \mathsf D \big( \langle \alpha \rangle \big) = 4$, implying that $|S_{G_0}| \ge 1$.
Since $\ord (\tau) = 4$ and $\tau^{2} = \alpha^{2}$, we infer that $|S_{G_0}|$ is even.

\smallskip
\noindent
{\bf CASE 1.} $|S_{G_0}| = 2$.

\smallskip
Then, by changing generators if necessary, we may assume that $S = T_1 \bdot \tau \bdot T_2 \bdot \alpha^{x}\tau$ with $1_{Q_8} = \sigma (T_1) \tau \sigma (T_2) \alpha^{x}\tau$ for some $T_1, T_2 \in \mathcal F \big( \langle \alpha \rangle \big)$ and $x \in [0,3]$.

If $T_2$ is trivial, then $T_1 \bdot \alpha^{2-x} \in \mathcal A \big( \langle \alpha \rangle \big)$ of length 4, and it follows that $T_1 = (\alpha^{i})^{[3]}$ and $2-x \equiv j \pmod{4}$ for some odd $j \in [1,3]$.
Since $j$ is odd, we have that $x = j$, and so
\begin{equation*}
  S = (\alpha^{j})^{[3]} \bdot \tau \bdot \alpha^{j}\tau \,\, \mbox{ with } \,\, j \in \{ 1, 3 \} \,.
\end{equation*}

If $T_1$ is trivial, then the same argument shows that $T_2 = (\alpha^{i})^{[3]}$ and $2 + x \equiv j \pmod{4}$ for some odd $j \in [1,3]$, whence we obtain that
\begin{equation*}
  S = (\alpha^{j})^{[3]} \bdot \tau \bdot \alpha^{j+2}\tau \,\, \mbox{ with } \,\, j \in \{ 1, 3 \} \,.
\end{equation*}

Thus, we may assume that both $T_1$ and $T_2$ are non-trivial sequences.
Since $S \in \mathcal A (Q_8)$, it follows that all $T_1$ and $T_2$ are product-one free sequences.
Since $|S_{\langle \alpha\rangle}| = 3$, we may assume that $|T_1| = 1$ and $|T_2| = 2$.
Then, it is easy to see that $T_2 = (\alpha^{j})^{[2]}$ or $T_2 = \alpha^{j} \bdot \alpha^{2j}$ for some odd $j \in [1,3]$.

If $T_2 = (\alpha^{j})^{[2]}$, then $S = \alpha^{i} \bdot \tau \bdot (\alpha^{j})^{[2]} \bdot \alpha^{x}\tau$ for some $i \in [1,3]$.
Since $\alpha^{2}$ is a center element in $Q_8$, the product-one equation of $S$ implies that $1_{Q_8} = \alpha^{i} \alpha^{2j} \tau \alpha^{x}\tau = \alpha^{i+2j-x+2}$.
Because $j$ is odd, it follows that $i \equiv x \pmod{4}$.
If $i = x = 2$, then $S = \big( (\alpha^{j})^{[2]} \bdot \alpha^{i} \big) \bdot \big( \tau \bdot \alpha^{x}\tau \big)$ is not a minimal product-one sequence.
Thus, $i=x$ is odd, and if $i=x$ and $j$ are distinct, then $S = \big( (\alpha^{j}) \bdot \alpha^{i} \big) \bdot \big( \alpha^{j} \bdot \tau \bdot \alpha^{x}\tau \big)$ is not a minimal product-one sequence.
Hence, we must have that $i = x = j$ is odd, so that
\begin{equation*}
  S = (\alpha^{j})^{[3]} \bdot \tau \bdot \alpha^{j}\tau \,\, \mbox{ with } \,\, j \in \{1, 3 \} \,.
\end{equation*}

If $T_2 = \alpha^{j} \bdot \alpha^{2}$, then since $\alpha^{2}$ is a center element in $Q_8$, we infer that $1_{Q_8} = \alpha^{i} \alpha^{2} \tau \alpha^{j} \alpha^{x}\tau$ for some $i \in [1,3]$.
Since $S \in \mathcal A (Q_8)$, it follows that $\alpha^{i} \bdot \alpha^{2}$ is a product-one free sequence, so that $i \in [1,3]$ must be odd.
It follows by the product-one equation that $x \equiv i-j \pmod{4}$, and since $j \in [1,3]$ is also odd, we have that $x = 0$ if $i = j$, or $x = 2$ if $i \neq j$.
Hence,
\begin{equation*}
  S = \left\{ \begin{array}{ll}
  		  \alpha^{2} \bdot (\alpha^{j})^{[2]} \bdot \tau^{[2]} \vspace{5pt}\\
                   \alpha^{2} \bdot \alpha^{j} \bdot \alpha^{j+2} \bdot \tau \bdot \alpha^{2}\tau
                   \end{array} \right. \,\, \mbox{ with } \,\, j \in  \{ 1, 3 \} \,.
\end{equation*}

\smallskip
\noindent
{\bf CASE 2.} $|S_{G_0}| = 4$.

\smallskip
Then, by changing generators if necessary, we may assume that $S = \alpha^{i} \bdot \tau \bdot \alpha^{x}\tau \bdot \alpha^{y}\tau \bdot \alpha^{z}\tau$ for some $i \in [1,3]$ and $x, y, z \in [0,3]$ satisfying the product-one equation
\[
  \alpha^{i} \tau \alpha^{x}\tau \alpha^{y}\tau \alpha^{z}\tau = 1_{Q_8} \quad \mbox{ or } \quad \tau \alpha^{i} \alpha^{x}\tau \alpha^{y}\tau \alpha^{z}\tau = 1_{Q_8} \,.
\]

Suppose that $\alpha^{i} \tau \alpha^{x}\tau \alpha^{y}\tau \alpha^{z}\tau = 1_{Q_8}$.
Since $S \in \mathcal A (Q_8)$, it follows that $S' = \alpha^{i} \bdot \alpha^{2-x} \bdot \alpha^{2+y-z} \in \mathcal A \big( \langle \alpha \rangle \big)$, whence $S' = (\alpha^{j})^{[2]} \bdot \alpha^{2}$ for some odd $j \in [1,3]$.

If $i = 2$ and $2-x \equiv 2+y-z \equiv j \pmod{4}$, then since $j$ is an odd, we obtain that $x = j$.
Thus, $y \equiv z-j \pmod{4}$, and so $S = \alpha^{2} \bdot \tau \bdot \alpha^{j}\tau \bdot \alpha^{z-j}\tau \bdot \alpha^{z}\tau$.
If $z=0$ (or $z = 2$), then for each odd $j \in [1,3]$,
\[
  S = \big( \alpha^{2} \bdot \tau \bdot \alpha^{z}\tau \big) \bdot \big( \alpha^{j}\tau \bdot \alpha^{z-j}\tau \big) \,\, \Big( \mbox{or } \,\, S = \big( \alpha^{2} \bdot \alpha^{j}\tau \bdot \alpha^{z-j}\tau \big) \bdot \big( \tau \bdot \alpha^{z}\tau \big) \Big)
\]
is not a minimal product-one sequence.
Thus, $z$ is also odd, and since $x = j$ is odd, we have either $z = j$ or $z \equiv j+2 \pmod{4}$.
Hence, we obtain that
\begin{equation} \label{eq:i=2}~
  S = \left\{ \begin{array}{ll}
  		  \alpha^{2} \bdot \tau^{[2]} \bdot (\alpha^{j}\tau)^{[2]} \vspace{5pt}\\
                   \alpha^{2} \bdot \tau \bdot \alpha^{2}\tau \bdot \alpha^{j}\tau \bdot \alpha^{j+2}\tau
                   \end{array} \right. \,\, \mbox{ with } \,\, j \in  \{ 1, 3 \} \,.
\end{equation}

If $2-x \equiv 2 \pmod{4}$ and $i \equiv 2+y-z \equiv j \pmod{4}$, then $x = 0$, $i = j$, and $y \equiv z+j+2 \pmod{4}$.
If $z = 2$, then $S = \big( \alpha^{j} \bdot \tau \bdot \alpha^{z+j+2}\tau \big) \bdot \big( \tau \bdot \alpha^{z}\tau \big)$ is not a minimal product-one sequence.
If $z$ is odd and $z \equiv j+2 \pmod{4}$, then $S = \big( \alpha^{j} \bdot \tau \bdot \alpha^{z}\tau \big) \bdot \big( \tau \bdot \alpha^{z+j+2}\tau \big)$ is not a minimal product-one sequence.
Thus, we have either $z = 0$ or $z = j$ (in this case, $z+j+2 \equiv 2j+2 \equiv 0 \pmod{4}$), and hence
\begin{equation*}
  S = \left\{ \begin{array}{ll}
  		  \tau^{[3]} \bdot \alpha^{j} \bdot \alpha^{j+2}\tau \vspace{5pt}\\
                   \tau^{[3]} \bdot \alpha^{j} \bdot \alpha^{j}\tau
                   \end{array} \right. \,\, \mbox{ with } \,\, j \in  \{ 1, 3 \} \,.
\end{equation*}

If $2+y-z \equiv 2 \pmod{4}$ and $i \equiv 2-x \equiv j \pmod{4}$, then since $j$ is odd, it follows that $y = z$, $i=j =x$.
If $z = 2$, then $S = \big( \alpha^{j} \bdot \alpha^{j}\tau \bdot \alpha^{z}\tau \big) \bdot \big( \tau \bdot \alpha^{z}\tau \big)$ is not a minimal product-one sequence.
If $z$ is odd and $z \equiv j+2 \pmod{4}$, then $S = \big( \alpha^{j} \bdot \tau \bdot \alpha^{z}\tau \big) \bdot \big( \alpha^{j}\tau \bdot \alpha^{z}\tau \big)$ is not a minimal product-one sequence.
Thus, we have either $z = 0$ or $z = j$, and hence
\begin{equation*}
  S = \left\{ \begin{array}{ll}
  		  \tau^{[3]} \bdot \alpha^{j} \bdot \alpha^{j}\tau \vspace{5pt}\\
                   (\alpha^{j}\tau)^{[3]} \bdot \alpha^{j} \bdot \tau
                   \end{array} \right. \,\, \mbox{ with } \,\, j \in  \{ 1, 3 \} \,.
\end{equation*}

Suppose now that $\alpha^{-i} \tau \alpha^{x}\tau \alpha^{y}\tau \alpha^{z}\tau = 1_{Q_8}$.
Since $S \in \mathcal A (Q_8)$, it follows that $S'' = \alpha^{-i} \bdot \alpha^{2-x} \bdot \alpha^{2+y-z} \in \mathcal A \big( \langle \alpha \rangle \big)$, and hence $S'' = (\alpha^{j})^{[2]} \bdot \alpha^{2}$ for some odd $j \in [1,3]$.

If $-i \equiv 2 \pmod{4}$, then $i = 2$, and hence $S$ has the same structure as in (\ref{eq:i=2}).

If $2-x \equiv 2 \pmod{4}$ and $-i \equiv 2+y-z \equiv j \pmod{4}$, then since $j$ is odd, we obtain that $x = 0$, $i \equiv j+2 \pmod{4}$, and $y \equiv z+j+2 \pmod{4}$.
If $z = 2$, then $S = \big( \alpha^{j+2} \bdot \tau \bdot \alpha^{z+j+2}\tau \big) \bdot \big( \tau \bdot \alpha^{z}\tau \big)$ is not a minimal product-one sequence.
If $z$ is odd and $z \equiv j+2 \pmod{4}$, then $S = \big( \alpha^{j+2} \bdot \tau \bdot \alpha^{z}\tau \big) \bdot \big( \tau \bdot \alpha^{z+j+2}\tau \big)$ is not a minimal product-one sequence.
Thus, we have either $z = 0$ or $z = j$ (in this case, $z+j+2 \equiv 2j+2 \equiv 0 \pmod{4}$), and hence
\begin{equation*}
  S = \left\{ \begin{array}{ll}
  		  \tau^{[3]} \bdot \alpha^{j+2} \bdot \alpha^{j+2}\tau \vspace{5pt}\\
                   \tau^{[3]} \bdot \alpha^{j+2} \bdot \alpha^{j}\tau
                   \end{array} \right. \,\, \mbox{ with } \,\, j \in  \{ 1, 3 \} \,.
\end{equation*}

If $2+y-z \equiv 2 \pmod{4}$ and $-i \equiv 2-x \equiv j \pmod{4}$, then since $j$ is odd, we obtain that $y = z$, $i \equiv j+2 \pmod{4}$, and $x=j$.
If $z = 2$, then $S = \big( \alpha^{j+2} \bdot \alpha^{j}\tau \bdot \alpha^{z}\tau \big) \bdot \big( \tau \bdot \alpha^{z}\tau \big)$ is not a minimal product-one sequence.
If $z$ is odd and $z \equiv j+2 \pmod{4}$, then $S = \big( \alpha^{j+2} \bdot \tau \bdot \alpha^{j+2}\tau \big) \bdot \big( \alpha^{j}\tau \bdot \alpha^{j+2}\tau \big)$ is not a minimal product-one sequence.
Thus, we have either $z = 0$ or $z = j$ (in this case, $z+j+2 \equiv 2j+2 \equiv 0 \pmod{4}$), and thence
\begin{equation*}
  S = \left\{ \begin{array}{ll}
  		  \tau^{[3]} \bdot \alpha^{j+2} \bdot \alpha^{j}\tau \vspace{5pt}\\
                   (\alpha^{j}\tau)^{[3]} \bdot \alpha^{j+2} \bdot \tau
                   \end{array} \right. \,\, \mbox{ with } \,\, j \in  \{ 1, 3 \} \,.
\end{equation*}
This complete the proof of Item 2.

\smallskip
3. Let $S \in \mathcal A (Q_8)$ with $|S| = 4$.
If $| S_{G_0} | = 0$, i.e., $S \in \mathcal F \big( \langle \alpha \rangle \big)$, then since $\mathsf D \big( \langle \alpha \rangle \big) = 4$, it follows that
\begin{equation*}
  S = \alpha^{[4]} \,.
\end{equation*}
Thus, since $\ord (\tau) = 4$ and $\tau^{2} = \alpha^{2}$, we may assume that $|S_{G_0}|$ is even.

\smallskip
\noindent
{\bf CASE 1.} $|S_{G_0}| = 2$.

\smallskip
Then, by changing generators if necessary, we may assume that $S = T_1 \bdot \tau \bdot T_2 \bdot \alpha^{x}\tau$ with $1_{Q_8} = \sigma (T_1) \tau \sigma (T_2)\alpha^{x}\tau$ for some $T_1, T_2 \in \mathcal F \big( \langle \alpha \rangle \big)$ and $x \in [0,3]$.

If $T_2$ is trivial, then $T_1 \bdot \alpha^{2-x} \in \mathcal A \big( \langle \alpha \rangle \big)$ of length 3, and so $T_1 \bdot \alpha^{2-x} = \alpha^{2} \bdot (\alpha^{j})^{[2]}$ for some odd $j \in [0,3]$.
If $2-x \equiv 2 \pmod{4}$, then $x = 0$, and hence
\begin{equation} \label{eq:x=0}~
  S = (\alpha^{j})^{[2]} \bdot \tau^{[2]} \,\, \mbox{ with } \,\, j \in \{ 1, 3 \} \,.
\end{equation}
If $2-x \equiv j \pmod{4}$, then since $j$ is odd, we have that $x = j$, and hence
\begin{equation*}
  S = \alpha^{2} \bdot \alpha^{j} \bdot \tau \bdot \alpha^{j}\tau \,\, \mbox{ with } \,\, j \in \{ 1, 3 \} \,.
\end{equation*}

If $T_1$ is trivial, then $T_2 \bdot \alpha^{x+2} \in \mathcal A \big( \langle \alpha \rangle \big)$ of length 3, and so $T_2 \bdot \alpha^{x+2} = \alpha^{2} \bdot (\alpha^{j})^{[2]}$ for some odd $j \in [1, 3]$.
If $x+2 \equiv 2 \pmod{4}$, then $x = 0$ and hence $S$ has the same structure as in (\ref{eq:x=0}).
If $x+2 \equiv j \pmod{4}$, then we obtain that
\begin{equation*}
  S = \alpha^{2} \bdot \alpha^{j} \bdot \tau \bdot \alpha^{j+2}\tau \,\, \mbox{ with } \,\, i \in \{ 1, 3 \} \,.
\end{equation*}

Hence, we suppose that both $T_1$ and $T_2$ are non-trivial, and since $S \in \mathcal A (Q_8)$, we infer that $T_1$ and $T_2$ are product-one free sequences.
This means that $T_1 = \alpha^{i}$ and $T_2 = \alpha^{j}$ for some $i, j \in [1,3]$, so that $S = \alpha^{i} \bdot \tau \bdot \alpha^{j} \bdot \alpha^{x}\tau$ with $i+2 \equiv j+x \pmod{4}$.

If both $i$ and $j$ are even, then $i = j = 2$ and $x = 2$, and so $S = \big( \alpha^{2} \bdot \alpha^{2} \big) \bdot \big( \tau \bdot \alpha^{2}\tau \big)$ is not a minimal product-one sequence.
If $i = 2$ and $j \in  \{ 1, 3 \}$ with $x \equiv j + 2 \pmod{4}$, then
\begin{equation*}
  S = \alpha^{2} \bdot \alpha^{j} \bdot \tau \bdot \alpha^{j+2}\tau \,\, \mbox{ with } \,\, j \in \{ 1, 3 \} \,.
\end{equation*}
If $j = 2$ and $i \in \{ 1, 3 \}$ with $x \equiv i \pmod{4}$, then
\begin{equation*}
  S = \alpha^{2} \bdot \alpha^{i} \bdot \tau \bdot \alpha^{i}\tau \,\, \mbox{ with } \,\, i \in \{ 1, 3 \} \,.
\end{equation*}
Suppose that both $i$ and $j$ are odd.
If $i \neq j$, then we must have that $x = 0$, and hence
\begin{equation*}
  S = \tau^{[2]} \bdot \alpha^{j} \bdot \alpha^{j+2} \,\, \mbox{ with } \,\, j \in  \{ 1, 3 \} \,.
\end{equation*}
If $i = j$, then $x = 2$, and thus
\begin{equation*}
  S = (\alpha^{j})^{[2]} \bdot \tau \bdot \alpha^{2}\tau \,\, \mbox{ with } \,\, j \in \{ 1, 3 \} \,.
\end{equation*}

\smallskip
\noindent
{\bf CASE 2.} $|S_{G_0}| = 4$.

\smallskip
Then, by changing generators if necessary, we may assume that $S = \tau \bdot \alpha^{x}\tau \bdot \alpha^{y}\tau \bdot \alpha^{z}\tau$ with $1_{Q_8} = \tau \alpha^{x}\tau \alpha^{y}\tau \alpha^{z}\tau$ for some $x, y, z \in [0,3]$.
Since $S \in \mathcal A (Q_8)$, it follows that $S' = \alpha^{2-x} \bdot \alpha^{y-z+2} \in \mathcal A \big( \langle \alpha \rangle \big)$, so that $S' = \alpha \bdot \alpha^{3}$ or $S' = \alpha^{2} \bdot \alpha^{2}$.

Suppose that $S' = \alpha \bdot \alpha^{3}$.
If $2-x \equiv 1 \pmod{4}$ and $y-z+2 \equiv 3 \pmod{4}$, then we have that $x = 1$ and $y \equiv 1+z \pmod{4}$.
If $z = 2$, then $S = \big( \tau \bdot \alpha^{z}\tau \big) \bdot \big( \alpha\tau \bdot \alpha^{1+z}\tau \big)$ is not a minimal product-one sequence.
Thus, we obtain that
\begin{equation*}
  S = \tau \bdot \alpha\tau \bdot \alpha^{z}\tau \bdot \alpha^{z+1}\tau \,\, \mbox{ with } \,\, z \in \{ 0, 1, 3 \} \,.
\end{equation*}
If $2-x \equiv 3 \pmod{4}$ and $y-z+2 \equiv 1 \pmod{4}$, then we have that $x = 3$ and $y \equiv z+3 \pmod{4}$.
If $z = 2$, then $S = \big( \tau \bdot \alpha^{z}\tau \big) \bdot \big( \alpha^{3}\tau \bdot \alpha\tau \big)$ is not a minimal product-one sequence.
Thus, we obtain that
\begin{equation*}
  S = \tau \bdot \alpha^{3}\tau \bdot \alpha^{z}\tau \bdot \alpha^{z+3}\tau \,\, \mbox{ with } \,\, z \in \{ 0, 1, 3 \} \,.
\end{equation*}

Suppose now that $S' = \alpha^{2} \bdot \alpha^{2}$, so that $2-x \equiv 2 \pmod{4}$ and $y-z+2 \equiv 2 \pmod{4}$.
Then, we have that $x = 0$ and $y = z$, and if $z = 2$, then $S = \big( \tau \alpha^{z}\tau \big) \bdot \big( \tau \alpha^{z}\tau \big)$ is not a minimal product-one sequence.
Thus, we obtain that
\begin{equation*}
  S = \tau^{[2]} \bdot (\alpha^{z}\tau)^{[2]} \,\, \mbox{ with } \,\, z \in \{ 0, 1, 3  \} \,.
\end{equation*}
This complete the proof of Item 3.

\smallskip
4. Let $S \in \mathcal A (Q_8)$ with $|S| = 3$.
If $|S_{G_0}| = 0$, i.e., $S \in \mathcal F \big( \langle \alpha \rangle \big)$, then we obtain that
\[
  S = \alpha^{2} \bdot ( \alpha^{j} )^{[2]} \,\, \mbox{ with } \,\, j \in \{ 1, 3 \} \,.
\]
Thus, since $\ord (\tau) = 4$ and $\tau^{2} = \alpha^{2}$, we only have the case when $|S_{G_0}| = 2$.
Then, by changing generators if necessary, we may assume that $S = T_1 \bdot \tau \bdot T_2 \bdot \alpha^{x}\tau$ with $\sigma (T_1) \tau \sigma (T_2) (\alpha^{x}\tau) = 1_{Q_8}$ for some $T_1, T_2 \in \mathcal F \big( \langle \alpha \rangle \big)$ and $x \in [0,3]$.
Since $|S| = 3$, it follows that either $T_1$ or $T_2$ is trivial.
Suppose first that $T_1$ is trivial.
Then, $S = \tau \bdot \alpha^{j} \bdot \alpha^{x} \tau$ for some $j \in [1,3]$ and $x \in [0,3]$, and the product-one equation ensures  that $-j - x + 2 \equiv 0 \pmod{4}$.
Thus, we obtain that $(j,x) \in \{ (1,1), (2,0), (3,3) \}$, and hence
\[
  S = \alpha^{j} \bdot \tau \bdot \alpha^{2-j}\tau \,\, \mbox{ with } \,\, j \in [1,3] \,.
\]
Suppose now that $T_2$ is trivial.
Then, $S = \alpha^{j} \bdot \tau \bdot \alpha^{x}\tau$ for some $i \in [1,3]$ and $x \in [0,3]$, and the product-one equation ensures that $j - x + 2 \equiv 0 \pmod{4}$.
Thus, we obtain that $(j,x) \in \{ (1,3), (2,0), (3,1) \}$, whence
\[
  S = \alpha^{j} \bdot \tau \bdot \alpha^{j+2}\tau \,\, \mbox{ with } \,\, j \in [1,3] \,.
\]
This complete the proof of Item 4.
\end{proof}

\smallskip
\begin{lemma} \label{lem:max}~
Let $U, V \in \mathcal A (Q_8)$.
\begin{enumerate}
\item $\max \mathsf L (U \bdot V) \le \min \{ |U|, |V| \}$.

\smallskip
\item If $|U|, |V| \in \{ 5, 6 \}$, then $\max \mathsf L (U \bdot V) \ge 3$.
\end{enumerate}
\end{lemma}

\begin{proof}
1. Let $U \bdot V = W_1 \bdot \ldots \bdot W_m$, where $m = \max \mathsf L (U \bdot V)$ and $W_1, \ldots, W_m  \in \mathcal A (Q_8)$.
Since $1_{Q_8}$ is a prime element in $\mathcal B (Q_8)$, we may assume that $U$, $V$, and $W_1, \ldots, W_m$ have their length at least 2.
Thus, if $m=2$, then the assertion follows directly, and so we assume that $m \ge 3$.
We set $U = U_1 \bdot \ldots \bdot U_m$ and $V = V_1 \bdot \ldots \bdot V_m$ with $U_i \bdot V_i = W_i$ for $i \in [1,m]$.
If $U_i$ and $V_i$ are non-trivial for all $i \in [1,m]$, then $m \le \sum_{i \in [1,m]} |U_i| = |U|$ and $m \le \sum_{i \in [1,m]} |V_i| = |V|$, and thus the assertion follows.

Suppose that there exists some $j \in [1,m]$ such that either $U_j$ or $V_j$ is trivial.
Then, after renumbering if necessary, we may assume that $U_1$ is a trivial sequence, so that $V_1 = W_1$.

We first suppose that $V_i$ is non-trivial for all $i \in [2,m]$.
Since $V \in \mathcal A (Q_8)$, it follows that
\[
  2 + (m-1) \le |V_1| + \small{\sum}_{i \in [2,m]} |V_i| = |V| \le \mathsf D (Q_8) = 6 \,,
\]
and thus $m \le 5$.
If $\min \{ |U|, |V| \} \ge 5$, then the assertion follows, whence we may assume that $\min \{ |U|, |V| \} \le 4$.
If $|V| = \min \{ |U|, |V| \} \le 4$, then $2 + (m-1) \le |V_1| + \sum_{i \in [2,m]} |V_i| = |V| \le 4$, and so $m = 3$.
In this case, we must have $|V| \ge 3$, for otherwise $m = 1$, a contradiction, and hence, $m = 3 \le |V| = \min \{ |U|, |V| \}$.
Now we consider the case $|U| = \min \{ |U|, |V| \} \le 4$.

Suppose that $m \ge 4$.
Then, $5 \le 2 + (m-1)  \le |V_1| + \sum_{i \in [2,m]} |V_i| = |V| \le 6$.
If $|V| = 6$, then in view of (\ref{eq:Da6}), $V = \alpha^{[4]} \bdot \tau^{[2]}$ for some generators $\alpha, \tau \in Q_8$.
Thus, it follows by Theorem~\ref{thm:Davenport5} that any product-one subsequence of $V$ must be of length 4, and hence $|V_1| = 4$.
Since $m \ge 4$, it follows that $|V| \ge 4 + 3 = 7  \gneq 6 = \mathsf D (Q_8)$, contradicting that $V \in \mathcal A (Q_8)$.
Thus, $|V| = 5$, and since $V_j$ is non-trivial for all $j \in [2,m]$, we obtain that $m = 4$, so that $|V_1| = 2$ and $|V_2| = |V_3| = |V_4| = 1$.
This means that $V$ is a minimal product-one sequence of length 5 having a product-one subsequence of length 2.
Then, in view of (\ref{eq:Da5}), we only have that $V = \alpha^{2} \bdot \alpha \bdot \alpha^{-1} \bdot \tau \bdot \tau^{-1}$ for some generators $\alpha, \tau \in Q_8$.
If $|U| \le 3$, then we obtain that either $U = \alpha^{2} \bdot \alpha \bdot \alpha^{-1}$ or $U = \alpha^{2} \bdot \tau \bdot \tau^{-1}$, but in view of (\ref{eq:Da3}), both are not minimal product-one sequences.
Thus, $|U| = 4$, and $m = 4 = |U| = \min \{ |U|, |V| \}$, whence the assertion follows.

Suppose that $m = 3$.
Assume to the contrary that $|U| = 2$, i.e., $U = g \bdot g^{-1}$ for some $g \in Q_8$.
Then, the same argument as used in the previous case shows that $|V| \in [4,6]$.
If $|V| = 6$, then in view of (\ref{eq:Da6}), $V = \alpha^{[4]} \bdot \tau^{[2]}$ for some generators $\alpha, \tau \in Q_8$.
In this case, since any product-one subsequence of $V$ has length 4, it follows that either $V_1 = \alpha^{[4]}$ and $V_2 = V_3 = \tau$, or that $V_1 = \alpha^{[2]} \bdot \tau^{[2]}$ and $V_2 = V_3 = \alpha$.
Since $|U| = 2$, we have that $U = \tau^{-1} \bdot \tau^{-1}$ or $U = \alpha^{-1} \bdot \alpha^{-1}$, but both are not product-one sequences, a contradiction.
If $|V| = 5$, then in view of (\ref{eq:Da5}), any product-one subsequence of $V$ must have length 2 or 3.
If $V$ has a product-one subsequence of length 2, then $V = \alpha^{2} \bdot \alpha \bdot \alpha^{-1} \bdot \tau \bdot \tau^{-1}$ is the only case for some generators $\alpha, \tau \in Q_8$, which ensures that either $V_1 = \alpha \bdot \alpha^{-1}$ or $V_1 = \tau \bdot \tau^{-1}$.
Since $|U| = 2$, we may assume that $|W_2| = 3$ and $|W_3| = 2$.
If $\alpha^{2} \mid W_2$, then in view of (\ref{eq:Da3}), $g \bdot g^{-1} = \tau \bdot \tau^{-1}$ or $g \bdot g^{-1} = \alpha \bdot \alpha^{-1}$.
This implies that $W_3 \in \{ \tau^{[2]}, (\tau^{-1})^{[2]}, \alpha^{[2]}, (\alpha^{-1})^{[2]} \}$, but in either case, $W_3$ is not product-one, a contradiction.
If $\alpha^{2} \mid W_3$, then since $|W_3| = 2$, it follows that $W_3 = \alpha^{2} \bdot \alpha^{2}$, which ensures that $U = \alpha^{2} \bdot \alpha^{2}$.
However, $W_2 = \tau \bdot \tau^{-1} \bdot \alpha^{2}$ or $W_2 = \alpha \bdot \alpha^{-1} \bdot \alpha^{2}$, and both are not product-one sequence, a contradiction.
If $|V| = 4$, then since $m = 3$ and $W_1, W_2, W_3$ are of length at least 2, we infer that $|W_1| = |W_2| = |W_3| = 2$.
In view of (\ref{eq:Da4}), $V = \alpha^{[2]} \bdot \tau \bdot \tau^{-1}$ is the only case, which implies that $U = \alpha^{-1} \bdot \alpha^{-1}$, a contradiction.
Therefore, we must have that $|U| \ge 3$, and $m = 3 \le |U| = \min \{ |U|, |V| \}$, whence the assertion follows.

We now assume that $V_i$ is trivial for some $i \in [2,m]$.
Then, after renumbering if necessary, we may assume that $V_2$ is trivial, so that $U_2 = W_2$.
Then, $U = W_2 \bdot U_3 \bdot \ldots \bdot U_m$ and $V = W_1 \bdot V_3 \bdot \ldots \bdot V_m$.
If $U_j$ and $V_j$ are non-trivial for all $j \in [3,m]$, then since $W_1, W_2 \in \mathcal A (Q_8)$, it follows that
\[
  2 + (m-2) \le |W_2| + \small{\sum}_{i \in [3,m]} |U_i| = |U| \quad \mbox{ and } \quad 2 + (m-2) \le |W_1| + \small{\sum}_{i \in [3,m]} |V_i| = |V| \,,
\]
whence the assertion follows.
If there exists some $j \in [3,m]$ such that $U_j$ or $V_j$ is trivial, then after renumbering if necessary, we may assume that $U_3$ is trivial, so that $V_3 = W_3$.
Since $V \in \mathcal A (Q_8)$, we also obtain that $m \le 5$, so that we may assume that $\min \{ |U|, |V| \} \le 4$.
If $|V| = \min \{ |U|, |V| \}$, then since $|V| \ge 4$, it follows that $|V| = 4$ and $V = W_1 \bdot W_3$, a contradiction to $V \in \mathcal A (Q_8)$.
Suppose that $|U| = \min \{ |U|, |V| \} \le 4$.
If $U = W_2$, then $V = W_1 \bdot W_3 \bdot W_4 \bdot \ldots \bdot W_m$, a contradiction to $V \in \mathcal A (Q_8)$.
Thus, we must have that $W_2$ is a proper product-one subsequence of $U$.
Since $|U| \le 4$, in views of (\ref{eq:Da4})--(\ref{eq:Da3}), we must have that $|U| = 4$ and either $U \bdot W^{[-1]}_2 = \alpha \bdot \alpha$ or $U \bdot W^{[-1]}_2 = \alpha^{2}$.
Since $V$ has at least 2 proper product-one subsequences, in views of (\ref{eq:Da6})--(\ref{eq:Da4}), $|V| = 5$ and $V = \alpha^{2} \bdot \alpha \bdot \alpha^{-1} \bdot \tau \bdot \tau^{-1}$ is the only case.
So, we may assume that $W_1 = \alpha \bdot \alpha^{-1}$ and $W_3 = \tau \bdot \tau^{-1}$.
In this case, we must have that $U \bdot V = W_1 \bdot W_2 \bdot W_3 \bdot W_4$ with $W_4 = \big( U \bdot W^{[-1]}_2 \big) \bdot \alpha^{2}$.
Thus, $m = 4 = |U| = \min \{ |U|, |V| \}$, and hence the assertion follows.

\smallskip
2. We show that $\mathsf L (U \bdot V)$ contains some integer $k \ge 3$, and we distinguish two cases.

\smallskip
\noindent
{\bf CASE 1.} $|U| = 6$.

\smallskip
In view of (\ref{eq:Da6}), there exist generators $\alpha, \tau \in Q_8$ such that $U = \alpha^{[4]} \bdot \tau^{[2]}$.

Suppose that $|V| = 6$.
Then, again in view of (\ref{eq:Da6}), $V = (\alpha')^{[4]} \bdot (\tau')^{[2]}$ for suitable generators $\alpha', \tau' \in Q_8 = \langle \alpha, \tau \rangle$.
For any choice of $\alpha', \tau'$, in view of (\ref{eq:Da4}), we obtain that
\[
  U \bdot V = \alpha^{[4]} \bdot (\alpha')^{[4]} \bdot \big( \tau^{[2]} \bdot (\tau')^{[2]} \big) \,,
\]
and thus $3 \in \mathsf L (U \bdot V)$.

Suppose that $|V| = 5$.
Then, by Theorem~\ref{thm:Davenport5}, $V$ has one of the forms listed in (\ref{eq:Da5}) for suitable generators $\alpha', \tau' \in Q_8 = \langle \alpha, \tau \rangle$.
For any choice of $\alpha', \tau'$, in view of (\ref{eq:Da4})--(\ref{eq:Da3}), we obtain the following cases:
\begin{itemize}
\item If $V = (\alpha')^{[3]} \bdot \tau' \bdot \alpha'\tau'$ (or $V = (\alpha')^{[3]} \bdot \tau' \bdot (\alpha')^{3}\tau'$), then we obtain that
           \[
             U \bdot V = \alpha^{[4]} \bdot \big( \tau^{[2]} \bdot (\alpha')^{[2]} \big) \bdot \big( \alpha' \bdot \tau' \bdot \alpha'\tau' \big) \,\, \Big( \mbox{or } U \bdot V = \alpha^{[4]} \bdot \big( \tau^{[2]} \bdot (\alpha')^{[2]} \big) \bdot \big( \alpha' \bdot \tau' \bdot (\alpha')^{3}\tau' \big) \Big) \,,
           \]
           and thus $3 \in \mathsf L (U \bdot V)$.

\smallskip
\item If $V = \alpha^{2} \bdot (\alpha')^{[2]} \bdot (\tau')^{[2]}$, then we obtain that
           \[
             U \bdot V = \alpha^{[4]} \bdot \big( \alpha^{2} \bdot \tau^{[2]} \big) \bdot \big( (\alpha')^{[2]} \bdot (\tau')^{[2]} \big) \,,
           \]
           and so $3 \in \mathsf L (U \bdot V)$.

\smallskip
\item If $V = \alpha^{2} \bdot \alpha' \bdot (\alpha')^{-1} \bdot \tau \bdot (\tau')^{-1}$, then we obtain that
           \[
             U \bdot V = \alpha^{[4]} \bdot \big( \alpha^{2} \bdot \tau^{[2]} \big) \bdot \big( \alpha' \bdot (\alpha')^{-1} \big) \bdot \big( \tau' \bdot (\tau')^{-1} \big) \,,
           \]
           and hence $3 \in \mathsf L (U \bdot V)$.
\end{itemize}

\smallskip
\noindent
{\bf CASE 2.} $|U| = 5$.

\smallskip
Then, by Theorem~\ref{thm:Davenport5}, there exist generators $\alpha, \tau \in Q_8$ such that $U$ has one of the forms listed in (\ref{eq:Da5}).
If $\alpha^{[3]} \mid U$, then in view of (\ref{eq:Da3}), we set $U = \alpha^{[2]} \bdot U'$, where $U' =  \alpha \bdot \tau \bdot \alpha^{x}\tau \in \mathcal A (Q_8)$ for $x \in \{ 1, 3 \}$.
If $\alpha^{2} \mid U$, then in view of (\ref{eq:Da4}), we also set $U = \alpha^{2} \bdot U''$, where $U'' = \alpha^{[2]} \bdot \tau^{[2]} \in \mathcal A (Q_8)$ or $U'' = \big( \alpha \bdot \alpha^{-1} \big) \bdot \big( \tau \bdot \tau^{-1} \big) \in \mathcal B (Q_8)$ is a product of two atoms.

By swapping the role between $U$ and $V$, it suffices to consider the case when $|V| = 5$.
Hence, Theorem~\ref{thm:Davenport5} ensures again that $V$ has one of the forms listed in (\ref{eq:Da5}) for suitable generators $\alpha', \tau' \in Q_8 = \langle \alpha, \tau \rangle$.
We distinguish two cases.

\smallskip
\noindent
{\bf SUBCASE 2.1.} $\alpha^{[3]} \mid U$.

\smallskip
If $(\alpha')^{[3]} \mid V$, then we can also write $V = (\alpha')^{[2]} \bdot V'$, and so in view of (\ref{eq:Da4}),
\[
  U \bdot V = \big( \alpha^{[2]} \bdot (\alpha')^{[2]} \big) \bdot U' \bdot V' \,,
\]
implying that $3 \in \mathsf L (U \bdot V)$.
If $\alpha^{2} \mid V_2$, then we also write $V = \alpha^{2} \bdot V''$, and hence in view of (\ref{eq:Da3}),
\[
  U \bdot V = \big( \alpha^{2} \bdot \alpha^{[2]} \big) \bdot U' \bdot V'' \,,
\]
forcing that $\mathsf L ( U \bdot V)$ contains 3 or 4.

\smallskip
\noindent
{\bf SUBCASE 2.2.} $\alpha^{2} \mid U$.

\smallskip
By swapping the role between $U$ and $V$, it suffices to consider the case when $\alpha^{2} \mid V$.
Then, we obtain that
\[
  U \bdot V = \big( \alpha^{2} \bdot \alpha^{2} \big) \bdot U'' \bdot V'' \,,
\]
implying that $\mathsf L (U \bdot V)$ contains some integer $k \in \{ 3, 4, 5 \}$.
\end{proof}

\smallskip
\begin{proposition} \label{pro:c2}~
We have $\daleth (Q_8) = 4$.
\end{proposition}

\begin{proof}
Since $Q_8$ has an element of order 4, we obtain $\daleth (Q_8) \ge 4$, whence it suffices to show that $\daleth (Q_8) \le 4$.
Assume to the contrary that there exist $U, V \in \mathcal A (Q_8)$ such that $\min \big( \mathsf L (U \bdot V) \setminus \{ 2 \} \big) \ge 5$.
Let $z \in \mathsf Z (U \bdot V)$ with $|z| = \min \big( \mathsf L (U \bdot V) \setminus \{ 2 \} \big)$.
Then, by Lemma~\ref{lem:max}.1, we obtain that
\[
  5 \, \le |z| \, \le \, \max \mathsf L (U \bdot V) \, \le \, \min \{ |U|, |V| \} \, \le \, \max \{ |U|, |V| \} \, \le \, \mathsf D (Q_8) \, = \, 6 \,,
\]
and hence $|U|, |V| \in \{5,6\}$.

Suppose that $|z| = 6$, and let $z = W_1 \bdot \ldots \bdot W_6$ for some $W_i \in \mathcal A (Q_8)$ with $|W_i| \ge 2$ for all $i\in [1,6]$.
Then, $12 \le \sum_{i \in [1,6]} |W_i| = |U| + |V| \le 12$ implies that $|U| = |V| = 6$, and hence $V = U^{-1}$ (see \cite[Lemma 4.2]{Oh19}).
In view of (\ref{eq:Da6}), there exist generators $\alpha, \tau \in Q_8$ such that
\[
  U = \alpha^{[4]} \bdot \tau^{[2]} \quad \mbox{ and } \quad V = (\alpha^{-1})^{[4]} \bdot (\tau^{-1})^{[2]} \,.
\]
In view of (\ref{eq:Da4}), $\alpha^{[2]} \bdot \tau^{[2]} \in \mathcal A (Q_8)$, so that we have a factorization
\[
  z_0 := \big( \alpha^{[2]} \bdot \tau^{[2]} \big) \bdot \big( (\alpha^{-1})^{[2]} \bdot (\tau^{-1})^{[2]} \big) \bdot \big( \alpha \bdot \alpha^{-1} \big)^{[2]} \in \mathsf Z (U \bdot V) \,\, \mbox{ with } \,\, |z_0| = 4 \,,
\]
a contradiction to the minimal length of $z$.

Suppose now that $|z| = 5$, and let $z = W_1 \bdot \ldots \bdot W_5$ for some $W_1, \ldots, W_5 \in \mathcal A (Q_8)$ with $|W_5| \ge |W_4| \ge \cdots \ge |W_1| \ge 2$.
If $|U| = |V| = 6$, then in view of (\ref{eq:Da6}), there exist generators $\alpha, \tau \in Q_8$ such that $U = \alpha^{[4]} \bdot \tau^{[2]}$, and also $V = (\alpha')^{[4]} \bdot (\tau')^{[2]}$ for some suitable generators $\alpha', \tau' \in Q_8 = \langle \alpha, \tau \rangle$.
For any choice of $\alpha', \tau'$, in view of (\ref{eq:Da4},
\[
  z_0 := \alpha^{[4]} \bdot (\alpha')^{[4]} \bdot \big( \tau^{[2]} \bdot (\tau')^{[2]} \big) \in \mathsf Z (U \bdot V) \,\, \mbox{ with } \,\, |z_0| = 3 \,,
\]
again a contradiction.
If $|U| = |V| = 5$, then $V = U^{-1}$, and by Theorem~\ref{thm:Davenport5}, there exist generators $\alpha, \tau \in Q_8$ such that $U$ has one of the forms listed in (\ref{eq:Da5}).

If $U = \alpha^{[3]} \bdot \tau \bdot \alpha^{x}\tau$ for $x \in \{ 1, 3 \}$, then $V = (\alpha^{-1})^{[3]} \bdot \tau^{-1} \bdot (\alpha^{x}\tau)^{-1}$, and in view of (\ref{eq:Da3}),
\[
  z_0 := \big( \alpha \bdot \tau \bdot \alpha^{x}\tau) \big) \bdot \big( \alpha^{-1} \bdot \tau^{-1} \bdot (\alpha^{x}\tau)^{-1} \big) \bdot \big( \alpha \bdot \alpha^{-1} \big)^{[2]} \in \mathsf Z (U \bdot V) \,\, \mbox{ with } \,\, |z_0| = 4 \,,
\]
again a contradiction.

If $U = \alpha^{2} \bdot \alpha^{[2]} \bdot \tau^{[2]}$ (resp., $U = \alpha^{2} \bdot \alpha \bdot \alpha^{-1} \bdot \tau \bdot \tau^{-1}$), then $V = \alpha^{2} \bdot (\alpha^{-1})^{[2]} \bdot (\tau^{-1})^{[2]}$ (resp., $V = U$), and in view of (\ref{eq:Da3}),
\[
  z_0 := \big( \alpha^{2} \bdot \tau^{[2]} \big) \bdot \big( \alpha^{2} \bdot (\tau^{-1})^{[2]} \big) \bdot \big( \alpha \bdot \alpha^{-1} \big)^{[2]} \in \mathsf Z (U \bdot V) \,\, \mbox{ with } \,\, |z_0| = 4 \,,
\]
again a contradiction.

Thus, we may assume that $|U| = 6$ and $|V| = 5$.
In view of (\ref{eq:Da6}), there exist generators $\alpha, \tau \in Q_8$ such that $U = \alpha^{[4]} \bdot \tau^{[2]}$.
Then, $|W_1| = \cdots = |W_4| = 2$ and $|W_5| = 3$.
Since $U = \alpha^{[4]} \bdot \tau^{[2]}$, and $z$ has 4 atoms of length 2, it follows by Theorem~\ref{thm:Davenport5} that $V$ has one of the forms listed in (\ref{eq:Da5}) with generators $\alpha^{-1}$ and $\tau^{-1}$.

If $V = (\alpha^{-1})^{[3]} \bdot \tau^{-1} \bdot \alpha^{x}\tau^{-1}$ for $x \in \{ 1, 3 \}$, then $W_5 = \alpha \bdot \tau \bdot \alpha^{x}\tau^{-1}$, and in views of (\ref{eq:Da4})--(\ref{eq:Da3}),
\[
  z_0 := \big( \alpha^{[2]} \bdot \tau^{[2]} \big) \bdot \big( \alpha^{-1} \bdot \tau^{-1} \bdot \alpha^{x}\tau^{-1} \big) \bdot \big( \alpha \bdot \alpha^{-1} \big)^{[2]} \in \mathsf Z (U \bdot V) \,\, \mbox{ with } \,\, |z_0| = 4 \,,
\]
again a contradiction.

If $V = \alpha^{2} \bdot (\alpha^{-1})^{[2]} \bdot (\tau^{-1})^{[2]}$, then $W_5 = \alpha^{2} \bdot \alpha^{[2]}$, and in views of (\ref{eq:Da4})--(\ref{eq:Da3}),
\[
  z_0 := \big( \alpha^{[2]} \bdot \tau^{[2]} \big) \bdot \big( \alpha^{2} \bdot (\tau^{-1})^{[2]} \big) \bdot \big( \alpha \bdot \alpha^{-1} \big)^{[2]} \in \mathsf Z (U \bdot V) \,\, \mbox{ with } \,\, |z_0| = 4 \,,
\]
again a contradiction.
\end{proof}

\smallskip
\begin{theorem} \label{thm:Ca}~
We have that $\mathsf {Ca} (Q_8) = [2,4]$, $\daleth^{*} (Q_8) = [3,4]$, and $\Delta (Q_8) = [1,2]$.
\end{theorem}

\begin{proof}
It suffices to show that $\mathsf c (Q_8) \le 4$.
Indeed, in view of (\ref{eq:ine}), we obtain that $\daleth (Q_8) = \mathsf c (Q_8) = 4$ and $\max \Delta (Q_8) = 2$, and since $Q_8$ satisfies Property {\bf P} by Theorem~\ref{thm:P}, the assertion follows by Theorem~\ref{thm:Int}.

Assume to the contrary that $\mathsf c (Q_8) \ge 5$.
Let $S \in \mathcal B (Q_8)$ with minimal length such that there exist two factorizations $z, z' \in \mathsf Z (S)$ with $|z| \ge |z'|$ and $|z| + |z'|$ maximal for which there is no 4-chain of factorizations between $z$ and $z'$.
Then, $|z| \ge 5$, and since $1_{Q_8} \in \mathcal B (Q_8)$ is a prime element, we may assume that
\[
  z = U_1 \bdot \ldots \bdot U_k \quad \mbox{ and } \quad z' = V_1 \bdot \ldots \bdot V_{\ell}
\]
for $U_1, \ldots, U_k, V_1, \ldots, V_{\ell} \in \mathcal A (Q_8)$ with $|U_i| \ge 2$ and $|V_j| \ge 2$ for all $i, j$ and $\gcd (z, z') = 1$.
If $\ell = 1$, then $|z| = k = 1$, which contradicts $|z| \ge 5$.
Hence, we may assume that $\ell \ge 2$.

Suppose that $\max \mathsf L (V_i \bdot V_j) \ge 3$ for some distinct $i, j \in [1,\ell]$.
We denote by $z^{i,j}$ the subsequence of $z'$ obtained by deleting two atoms $V_i$ and $V_j$.
Since $\daleth (Q_8) = 4$ by Proposition~\ref{pro:c2}, there exists $z_0 \in \mathsf Z (V_i \bdot V_j)$ with $3 \le |z_0| \le 4$ and
\begin{equation} \label{eq:4-chain}~
  \mathsf d (z', z_0 \bdot z^{i,j}) = \mathsf d (V_i \bdot V_j, z_0) = |z_0| \le 4 \,.
\end{equation}
Since $|z| + |z_0 \bdot z^{i,j}| \gneq |z| + |z'|$, the maximality of $|z| + |z'|$ ensures that there is a 4-chain of factorizations of $S$ between $z$ and $z_0 \bdot z^{i,j}$.
Moreover, in view of (\ref{eq:4-chain}), we infer that there exists a 4-chain of factorizations of $S$ concatenating $z$ and $z'$, a contradiction.
Thus, we obtain that
\begin{equation} \label{eq:2atoms}~
  \mathsf L ( V_i \bdot V_j ) = \{ 2 \} \,\, \mbox{ for all distinct } \,\, i, j \in [1,\ell] \,.
\end{equation}
If $\ell = 2$, then $U_1 \bdot \ldots \bdot U_k = V_1 \bdot V_2$, and it follows that $k \le \daleth (Q_8) = 4$, which again contradicts $|z| \ge 5$.
Thus, we further assume that $\ell \ge 3$, and after renumbering if necessary, we assume that $|V_1| \ge |V_2| \ge \cdots \ge |V_{\ell}| \ge 2$.
In view of (\ref{eq:2atoms}), by Lemma~\ref{lem:max}.2, there exists at most one atom of length at least 5 among $V_1, \ldots, V_{\ell}$, and hence we suppose that
\begin{equation} \label{eq:length}~
  |V_j| \le 4 \,\, \mbox{ for all } \,\, j \in [2,\ell] \,.
\end{equation}
We proceed with the following assertion.

\smallskip
\begin{enumerate}
\item[\namedlabel{itm:A1}{\bf A1}.] If there exists some $j \in [1,\ell]$ such that $V_j \mid U_1 \bdot \ldots \bdot U_{i}$ in $\mathcal B (Q_8)$ for some $i \in [1,k-1]$, then there exists a 4-chain of factorizations of $S$ between $z$ and $z'$.
\end{enumerate}

\smallskip
\begin{proof}[Proof of \ref{itm:A1}]
After renumbering if necessary, we may assume that $j = 1$, and there exists some $i \in [1,k-1]$ such that
\[
  U_1 \bdot \ldots \bdot U_i = V_1 \bdot W \,\, \mbox{ for some } \,\, W \in \mathcal B (Q_8) \,.
\]
By the minimality of $|S|$, there exists a 4-chain of factorizations $y_0, \ldots, y_r$ between $y_0 = U_1 \bdot \ldots \bdot U_i$ and $y_r = V_1 \bdot W$, and also a 4-chain of factorizations $w_0, \ldots, w_t$ between $w_0 = W \bdot U_{i+1} \bdot \ldots \bdot U_k$ and $w_t = V_2 \bdot \ldots \bdot V_{\ell}$.
Hence,
\[
  z = y_0 \bdot U_{i+1} \bdot \ldots \bdot U_k \,, \,\, y_1 \bdot U_{i+1} \bdot \ldots \bdot U_k \,,  \,\, \ldots \,\,, \,\, y_r \bdot U_{i+1} \bdot \ldots \bdot U_k = V_1 \bdot w_0 \,, \,\, V_1 \bdot w_1 \,, \,\, \ldots \,\,, \,\, V_1 \bdot w_t = z'
\]
is a 4-chain of factorizations of $S$ concatenating $z$ and $z'$.  \qedhere[\ref{itm:A1}]
\end{proof}

\smallskip
In view of \ref{itm:A1}, it suffices to show that there exists some $j \in [1,\ell]$ such that $V_j \mid U_1 \bdot \ldots \bdot U_i$ in $\mathcal B (Q_8)$ for some $i \in [1,k-1]$.
By renumbering if necessary, we suppose that $V_2 \mid U_1 \bdot \ldots \bdot U_{|V_2|}$ in $\mathcal F (Q_8)$.
If all terms in $V_2$ split completely into atoms $U_i$ for all $i \in [1,|V_2|]$, then since $V_2$ is product-one, we may write $V_2 = g_1 \bdot \ldots \bdot g_{|V_2|}$ with $g_1 \cdots g_{|V_2|} = 1_{Q_8}$ and $g_i \mid U_i$ for all $i \in [1,|V_2|]$.
Then, for each $i \in [1,|V_2|]$, the product-one equation of $U_i$ ensures that $g^{-1}_i \in \pi \big( U_i \bdot g^{[-1]}_i \big)$.
It follows that
\[
  1_{Q_8} = g^{-1}_{|V_2|} \cdots g^{-1}_1 \in \small{\prod}_{i \in [1,|V_2|]} \pi \big( U_i \bdot g^{[-1]}_i \big) \subseteq \pi \Big( {\small \prod}^{\bullet}_{i \in [1,|V_2|]} (U_i \bdot g^{[-1]}_i) \Big) = \pi \big( U_1 \bdot \ldots \bdot U_{|V_2|} \bdot V^{[-1]}_2 \big) \,.
\]
This means that $V_2 \mid U_1 \bdot \ldots \bdot U_{|V_2|}$ in $\mathcal B (Q_8)$, which, in view of~\ref{itm:A1}, yields a contradiction.
Hence, there must exist some $i \in [1,|V_2|]$ such that $U_i$ contains at least two terms from $V_2$.
If $|V_2| = 2$ and $V_2 \in \mathcal A (Q'_8)$, then $V_2$ consists of only non-identity central element.
Note that the multiplicity of non-identity central element in atoms of length at least 3 is at most 1.
Since $\gcd (z, z') = 1$, this allows us to assume that two terms in $V_2$ must split into distinct atoms, say $U_1$ and $U_2$.
Hence, the same argument shows that $V_2 \mid U_1 \bdot U_2$ in $\mathcal B (Q_8)$, which, in view of \ref{itm:A1}, again leads to a contradiction.
Therefore, we can exclude the case where $V_2 \in \mathcal A (Q'_8)$ of length 2.
Thus, regardless of the length of $V_2$, we assume by renumbering if necessary that
\[
  V_2 \mid U_1 \bdot \ldots \bdot U_{\varepsilon} \,\, \mbox{ in } \,\, \mathcal F (Q_8) \,,
\]
where $\varepsilon \in [1, |V_2|-1]$.
Then, $W = U_1 \bdot \ldots \bdot U_{\varepsilon} \bdot V^{[-1]}_2 \in \mathcal F (Q_8)$ with $\pi (W) \subseteq Q'_8$.
In view of \ref{itm:A1}, we further assume that $W \notin \mathcal B (Q_8)$, whence $\pi (W)$ is a singleton consisting of non-identity central element in $Q_8$.

If $|\pi (U_j)| = 2$ for some $j \in [\varepsilon+1,k]$, then we infer that $W' := W \bdot U_j \in \mathcal B (Q_8)$.
If $|\pi (U_i)| = 1$ for all $i \in [\varepsilon+1,k]$ and $|\pi (W \bdot U_j)| = 2$ for some $j \in [\varepsilon+1,k]$, then we infer that $W' := W \bdot U_j \in \mathcal B (Q_8)$.
In either case, we may assume by renumbering if necessary that $V_2 \mid U_1 \bdot \ldots \bdot U_{\varepsilon+1}$ in $\mathcal B (Q_8)$, which, in view of \ref{itm:A1}, yields a contradiction.

Suppose now that $|\pi (U_j)| = 1$ and $|\pi (W \bdot U_j)| = 1$ for all $j \in [\varepsilon+1,k]$.
We will use the following observation repeatedly.

\smallskip
\begin{enumerate}
\item[\namedlabel{itm:A2}{\bf A2}.] If $\varepsilon \le 2$ and $|W| = 1$, then there exists a maximal subgroup $H$ of $Q_8$ such that $U_{\varepsilon+1}, \ldots, U_k$, $V_1, V_3, \ldots, V_{\ell}$ are atoms over $H$ and at least one term in $U_{\varepsilon+1} \bdot \ldots \bdot U_k$ is a generator of $H$.
\end{enumerate}

\smallskip
\begin{proof}[Proof of \ref{itm:A2}]
Suppose that $\varepsilon \le 2$.
Since $U_1, \ldots, U_{\varepsilon}, V_2 \in \mathcal B (Q_8)$ and $W \notin \mathcal B (Q_8)$, it follows that $\pi (W)$ is a singleton consisting of the non-identity central element in $Q_8$.
If $|\pi (U_n \bdot U_m)| = 2$ for some $n, m \in [\varepsilon+1,k]$, say $U_{\varepsilon+1}$ and $U_{\varepsilon+2}$, then $1_{Q_8} = z z \in \pi (W) \pi (U_{\varepsilon+1} \bdot U_{\varepsilon+2}) \subseteq \pi (W \bdot U_{\varepsilon+1} \bdot U_{\varepsilon+2})$, and so $W' := W \bdot U_{\varepsilon+1} \bdot U_{\varepsilon+2} \in \mathcal B (Q_8)$, which ensures that $V_2 \mid U_1 \bdot \ldots \bdot U_{\varepsilon+2}$ in $\mathcal B (Q_8)$.
Since $\varepsilon+ 2 \le 4$ and $k \ge 5$, \ref{itm:A1} allows us to construct a 4-chain of factorizations of $S$ between $z$ and $z'$, a contradiction.
Thus, we have that $|\pi (U_m \bdot U_n)| = 1$ for all distinct $m, n \in [\varepsilon+1,k]$.
Since $|\pi (U_j)| = 1$ and $|\pi (W \bdot U_j)| = 1$ for all $j \in [\varepsilon+1,k]$, we infer that $U_{\varepsilon+1}, \ldots, U_k, V_1, V_3, \ldots V_{\ell}$ are sequences over an abelian subgroup $H$ of $Q_8$.
Since $Q_8$ is a minimal non-abelian group (see Proposition~\ref{pro:min}), $H$ can be considered as a maximal subgroup, which is actually cyclic group of order 4.
Let $Q'_8 = \{ 1_{Q_8}, z \}$.
Since any atom of length at least 3 can have at most 1 term of $z$ and $|W| = 1$, we infer that $|(U_1 \bdot \ldots \bdot U_{\varepsilon})_{Q'_8}| = |(V_2 \bdot W)_{Q'_8}| = t$ must be odd.
Hence, if $U_{\varepsilon+1}, \ldots, U_k, V_1, V_3, \ldots, V_{\ell}$ are all sequences over $Q'_8$, then $U_{\varepsilon+1} = \cdots = U_k = V_1 = V_3 = \cdots V_{\ell} = z^{[2]}$.
This means that the number of $z$ in $U_1 \bdot \ldots \bdot U_k$ is $2(k-\varepsilon) + t$, which is odd, whereas it is $2(\ell - 1) + (t-1)$, which is even, in $V_1 \bdot \ldots \bdot V_{\ell}$, leading a contradiction.
Thus, at least one term in $U_{\varepsilon+1} \bdot \ldots \bdot U_k$ is in $H \setminus Q'_8$, and since $H$ is cyclic, such term must be a generator of $H$. \qedhere[\ref{itm:A2}]
\end{proof}

Now we proceed with the case distinctions.

\medskip
\noindent
{\bf CASE 1.} $W$ consists of only non-identity central element $z \in Q'_8 \setminus \{ 1_{Q_8} \}$.

\smallskip
If $|V_2| = 2$, then $\varepsilon = 1$, and since $V_2 \notin \mathcal A (Q'_8)$, we have that $V_2 = g \bdot g^{-1}$ for some $g \in Q_8 \setminus Q'_8$, and since $\gcd (z,z') = 1$, it follows that $|U_1| \ge 3$.
Since $W = z^{[|W|]}$, in views of (\ref{eq:Da6})--(\ref{eq:Da3}), $|U_1| = 3$ and $U_1 = z \bdot V_2$, which implies that $\pi (U_1) = \{ z \}$, a contradiction to $U_1 \in \mathcal A (Q_8)$.

Thus, we have $3 \le |V_2| \le 4$, and since $|W|$ must be odd, it follows that $|W| \in \{ 1, 3, 5 \}$.
If $|W| = 5$, then since any atom of length at least 3 can have at most 1 term of non-identity central element, we must have that $|V_2| = 4$ and $\varepsilon = |V_2|-1=3$, which means that all $U_1, U_2, U_3$ contain at least one term from $V_2$.
By renumbering if necessary, we may assume that $U_1 = U_2 = z^{[2]}$ and $z \mid U_3$, and thus $z^{[2]} \mid V_2$, a contradiction to $V_2 \in \mathcal A (G)$ with $|V_2| = 4$.

Suppose now that $|W| = 1$.
Since $3 \le |V_2| \le 4$, we first consider the case $|V_2| = 4$.
Then, if $\varepsilon = 3$, then $V_2 \mid U_1 \bdot U_2 \bdot U_3$ implies that $6 \le \Sigma_{i \in [1,3]} |U_i| = |V_2| + |W| = 5$, a contradiction, whence we may assume that $\varepsilon \le 2$.
If $\varepsilon = 1$, then in view of (\ref{eq:Da5}), $V_2 = \alpha^{[2]} \bdot \tau^{[2]}$ and $U_1 = z \bdot V_2$ for some generators $\alpha, \tau \in Q_8$.
In view of \ref{itm:A2}, we infer that $U_2, \ldots, U_k, V_1, V_3, \ldots, V_{\ell}$ are sequences over a cyclic subgroup $H$ of $Q_8$ of order 4, and at least one term in $U_2 \bdot \ldots \bdot U_k$ is a generator of $H$.
Let $V_t$ be such that at least one term is a generator of $H$ for some $t \in [1,\ell]$ with $t \neq 2$.
If $V_t \mid U_2 \bdot \ldots \bdot U_k$ in $\mathcal F (Q_8)$, then as sequences over the abelian group $H$, the divisibility holds in $\mathcal B (Q_8)$, which, in view of \ref{itm:A1}, yields a contradiction.
Thus, $V_t = V_{t_1} \bdot V_{t_2}$ for some non-trivial sequences $V_{t_1}, V_{t_2} \in \mathcal F (H)$ with $V_{t_1} \mid U_1$ and $V_{t_2} \mid U_2 \bdot \ldots \bdot U_k$.
Since $|V_t| \le 4$, by renumbering if necessary, we may assume that $V_t \mid U_1 \bdot \ldots \bdot U_4$ in $\mathcal F (Q_8)$.
If we set $W' := U_1 \bdot \ldots \bdot U_4 \bdot V^{[-1]}_t$, then clearly $|\pi (W') | = 2$, and so $V_t \mid U_1 \bdot \ldots \bdot U_4$ in $\mathcal B (Q_8)$, which, in view of \ref{itm:A1}, yields a contradiction.
Hence, we consider the case $\varepsilon = 2$, and by renumbering if necessary, we set $|U_1| = 2$ and $|U_2| = 3$.
If $V_2$ splits equally into $U_1$ and $U_2$, then since $U_1 \mid V_2$, it follows that $W \mid U_2$ and $U_1 = h \bdot h^{-1}$ for some $h \in Q_8 \setminus Q'_8$, which, in view of (\ref{eq:Da4}), ensures that $V_2 = g^{[2]} \bdot h \bdot h^{-1}$ for some $g \in Q_8$ with $Q_8 = \langle g, h \rangle$, i.e., $U_2 = g^{[2]} \bdot z$.
In view of \ref{itm:A2}, we infer that $U_3, \ldots, U_k, V_1, V_3, \ldots, V_{\ell}$ are sequences over a cyclic subgroup $H$ of $Q_8$ of order 4, and at least one term in $U_3 \bdot \ldots \bdot U_k$ is a generator of $H$.
Let $V_t$ be such that at least one term is a generator of $H$ for some $t \in [1,\ell]$ with $t \neq 2$.
Suppose that $|V_t| = 4$ for some $t \in [1, \ell]$ with $t \neq 2$, i.e., $V_t = \alpha^{[4]}$ for a generator $\alpha$ of $H$.
By the same argument as used before, \ref{itm:A1} guarantees that $V_t = V_{t_1} \bdot V_{t_2}$ for some non-trivial subsequences $V_{t_1}, V_{t_2} \in \mathcal F (H)$ with $V_{t_1} \mid U_1 \bdot U_2$ and $V_{t_2} \mid U_3 \bdot \ldots \bdot U_k$.
According to the structures of $U_1$ and $U_2$, we obtain that $|V_{t_1}| \le 2$.
If $|V_{t_1}| = 1$, then $V_{t_2} = \alpha^{[3]}$.
If $V_{t_2}$ does not divide at most 2 atoms among $U_3, \ldots, U_k$, then since $V_t \nmid U_3 \bdot \ldots \bdot U_k$, we infer that $U_j \in \{ \alpha \bdot \alpha^{-1}, z \bdot z \}$ for all $j \in [3,k]$ and there are only 3 atoms having the form $\alpha \bdot \alpha^{-1}$, which means that the number of $\alpha$ in $W \bdot U_3 \bdot \ldots \bdot U_k$ must be 3.
Since $W \bdot U_3 \bdot \ldots \bdot U_k = V_1 \bdot V_3 \bdot \ldots \bdot V_{\ell}$ and $V_t = \alpha^{[4]}$, this is impossible.
Thus, we  have that $|V_{t_1}| = 2$.
In this case, $V_{t_1} \mid U_2$, and so we may assume by renumbering if necessary that $V_t \mid U_2 \bdot U_3 \bdot U_4$ in $\mathcal F (Q_8)$.
If we  set $W' := U_2 \bdot U_3 \bdot U_4 \bdot V^{[-1]}_t$, then clearly $| \pi (W' \bdot U_1)| = 2$, and thus $V_t \mid U_1 \bdot \ldots \bdot U_4$ in $\mathcal B (Q_8)$, which, in view of \ref{itm:A1}, yields a contradiction.
Suppose now that $|V_t| \le 3$.
By the same argument as used before, \ref{itm:A1} guarantees that $V_t = V_{t_1} \bdot V_{t_2}$ for some non-trivial sequences $V_{t_1}, V_{t_2} \in \mathcal F (H)$ with $V_{t_1} \mid U_1 \bdot U_2$ and $V_{t_2} \mid U_3 \bdot \ldots \bdot U_k$.
By renumbering and changing generators if necessary, we may assume that $V_t \mid U_1 \bdot U_2 \bdot U_3$ (if $|V_{t_1}| = 2$) or $V_t \mid U_2 \bdot U_3 \bdot U_4$ (if $|V_{t_1}| = 1$).
To be on the safe side, if we set $W' := U_1 \bdot U_2 \bdot U_3 \bdot U_4 \bdot V^{[-1]}_t$, then clearly $| \pi (W') | = 2$, and so $V_t \mid U_1 \bdot \ldots \bdot U_4$ in $\mathcal B (Q_8)$, which, in view of \ref{itm:A1}, yields a contradiction.
Hence, it remains to consider the case where $U_1$ contains only one term and $U_2$ contains exactly three terms from $V_2$.
Since $|U_2| = 3$, $U_2 \mid V_2$, and in view of (\ref{eq:Da4}), we obtain that $V_2 = z \bdot U_2$, and so $U_1 = z^{[2]}$.
In view of \ref{itm:A2}, we infer that $U_3, \ldots, U_k, V_1, V_3, \ldots, V_{\ell}$ are sequences over a cyclic subgroup $H$ of $Q_8$ of order 4, and at least one term in $U_3 \bdot \ldots \bdot U_k$ is a generator of $H$.
Since $U_1 = z^{[2]}$ and $\gcd (z, z') = 1$, there exists $t \in [1,\ell]$ with $t \neq 2$ such that $|V_t| = 3$, so that $V_t = z \bdot \alpha^{[2]}$ for a generator $\alpha$ of $H$.
By the same argument as used before, \ref{itm:A1} guarantees that $V_t = V_{t_1} \bdot V_{t_2}$ for some non-trivial sequences $V_{t_1}, V_{t_2} \in \mathcal F (H)$ with $V_{t_1} \mid U_1 \bdot U_2$ and $V_{t_2} \mid U_3 \bdot \ldots \bdot U_k$.
According to the structures of $U_1$, $U_2$, and $V_2$, it follows by renumbering if necessary that $V_t \mid U_1 \bdot U_3 \bdot U_4$, or $V_t \mid U_2 \bdot U_3 \bdot U_4$, or $V_t \mid U_1 \bdot U_2 \bdot U_3$ in $\mathcal F (Q_8)$.
To be on the safe side, we may assume that $V_t \mid U_1 \bdot \ldots \bdot U_4$ in $\mathcal F (Q_8)$.
If we set $W' := U_1 \bdot \ldots \bdot U_4 \bdot V^{[-1]}_t$, then since $U_2$ consists of non-commuting terms, $|\pi (W')| = 2$, and so $V_t \mid U_1 \bdot \ldots \bdot U_4$ in $\mathcal B (Q_8)$, which, in view of \ref{itm:A1}, yields a contradiction.

Now we consider the case $|V_2| = 3$.
If $\varepsilon = 2$, then $|U_1| = |U_2| = 2$, and since $|W| = 1$, we may assume that $U_1 = z^{[2]}$ and $z \mid V_2$.
In view of (\ref{eq:Da3}), $V_2 = z \bdot g^{[2]}$ for some generator $g$ of $Q_8$.
This implies that $U_2 = V_2 \bdot z^{[-1]} = g^{[2]}$, but $U_2$ is not a product-one sequence, a contradiction.
Thus, by renumbering if necessary, we may assume that $V_2 \mid U_1$, so that $U_1 = z \bdot V_2 \in \mathcal A (Q_8)$ of length 4.
In view of (\ref{eq:Da4}), $V_2 = \alpha \bdot \tau \bdot \alpha^{x}\tau$ for some generators $\alpha, \tau \in Q_8$ with $x \in \{1, 3\}$.
In view of \ref{itm:A2}, we infer that $U_2, \ldots, U_k, V_1, V_3, \ldots, V_{\ell}$ are sequences over a cyclic subgroup $H$ of $Q_8$ of order 4, and at least one term in $U_2 \bdot \ldots \bdot U_k$ is a generator of $H$.
By the same argument as used before, \ref{itm:A1} guarantees that there exists some $t \in [1,\ell]$ with $t \neq 2$ such that $V_t = V_{t_1} \bdot V_{t_2}$ for some non-trivial sequences $V_{t_1}, V_{t_2} \in \mathcal F (H)$ with $V_{t_1} \mid U_1$ and $V_{t_2} \mid U_2 \bdot \ldots \bdot U_k$.
Since $|V_t| \le 4$, we may assume by renumbering if necessary that $V_t \mid U_1 \bdot \ldots \bdot U_4$ in $\mathcal F (Q_8)$.
According to the structure of $U_1$, if we set $W' := U_1 \bdot \ldots \bdot U_4 \bdot V^{[-1]}_t$, then clearly $|\pi (W')| = 2$, and so $V_t \mid U_1 \bdot \ldots \bdot U_4$ in $\mathcal B (Q_8)$, which, in view of \ref{itm:A1}, yields a contradiction.

Suppose finally that $|W| = 3$.
If $|V_2| = 3$, then by renumbering if necessary, we may assume that $\varepsilon = 1$ and $|W| = 1$, because $\varepsilon = 2$ ensures that $U_1 \bdot U_2 = V_2 \bdot W$ with $z \mid U_1$ and $U_2 = z^{[2]}$, so that $V_2 \mid U_1$ in $\mathcal F (Q_8)$, and we have already shown that this case leads to a contradiction.
Thus, we assume that $|V_2| = 4$, and write $V_2 = g_1 \bdot \ldots \bdot g_4$ for some $g_1, \ldots, g_4 \in Q_8$.
If $\varepsilon \le 2$, then since $|V_2| = 4$ and $W = z^{[3]}$, then it follows by renumbering if necessary that $V_2 \mid U_1$ in $\mathcal F (Q_8)$, which allows us to assume that $\varepsilon = 1$ and $|W| = 1$, a case in which we have already obtained a contradiction.
Hence, it remains to consider the case $\varepsilon = 3$, which means that all $U_1, U_2, U_3$ contain at least one term from $V_2$.
If $g_i \in Q_8 \setminus Q'_8$ for all $i \in [1,4]$, then $|U_i| \ge 3$ for all $i \in [1,3]$, but $U_1 \bdot U_2 \bdot U_3 = V_2 \bdot W$ ensures that $9 \le \sum_{i \in [1,3]} |U_i| = |V_2| + |W| = 7$, a contradiction.
Thus, we suppose that $g_1 = z$ and $g_2, g_3, g_4 \in Q_8 \setminus Q'_8$.
By renumbering if necessary, we infer that $U_1 = z^{[2]}$, $z \mid U_i$ for all $i \in [2,3]$, some $U_i$ contains two terms from $V_2$, say $z \bdot g_2 \bdot g_3 \mid U_2$ and $z \bdot g_4 \mid U_3$.
Since $|V_2 \bdot W| = 7$, we obtain that $U_1 = z^{[2]}$, $U_2 = z \bdot g_2 \bdot g_3$, and $U_3 = z \bdot g_4$, and hence $g_4 = z \in Q'_8$, which ensures that $z^{[2]} \mid V_2$, a contradiction to $V_2 \in \mathcal A (Q_8)$ with $|V_2| = 4$.

\smallskip
\noindent
{\bf CASE 2.} $W$ contains at least one term from $Q_8 \setminus Q'_8$.

\smallskip
Then, since $|\pi(W \bdot U_j)| = 1$ for all $j \in [\varepsilon+1,k]$, it follows that $W, U_{\varepsilon+1}, \ldots, U_k$ are sequences over an abelian subgroup $H$.
Moreover, in views of (\ref{eq:Da6})--(\ref{eq:Da5}), we suppose that $|V_1| \le 4$, for otherwise there must be at least one $U_j$ such that $|\pi (W \bdot U_j )| = 2$, a contradiction.
From now on, by renumbering if necessary, we replace $V_2$ with $V_1$.
Since $U_1 \bdot \ldots \bdot U_{\varepsilon} = V_1 \bdot W$ and $W \notin \mathcal B (Q_8)$, we infer that $V_1 \notin \mathcal F (H)$, so that $V_1$ contains at least one term from $Q_8 \setminus H$.
By changing generators if necessary, we may assume that $Q_8 = \langle \alpha, \tau \mid \alpha^{4} = 1_{Q_8}, \tau^{2} = \alpha^{2}, \text{ and } \tau \alpha = \alpha^{-1} \tau \rangle$ and $H = \langle \alpha \rangle$, so that $|(V_1)_{Q_8 \setminus \langle \alpha \rangle}| \ge 2$ is even.
According to $|\pi (W \bdot U_j)| = 1$ for all $j \in [\varepsilon+1,k]$, we infer that $V_i \in \mathcal A \big( \langle \alpha \rangle \big)$ for all $i \in [2,\ell]$.
Since $W \in \mathcal F \big( \langle \alpha \rangle \big)$, each $U_1, \ldots, U_{\varepsilon}$ cannot contain odd terms from $Q_8 \setminus \langle \alpha \rangle$ in $V_1$.
Hence, it follows by renumbering if necessary that $(V_1)_{Q_8 \setminus \langle \alpha \rangle} \mid U_1$ or $(V_1)_{Q_8 \setminus \langle \alpha \rangle} \mid U_1 \bdot U_2$.

Suppose first that $(V_1)_{Q_8 \setminus \langle \alpha \rangle} \mid U_1$, so that $U_2, \ldots, U_k \in \mathcal A \big( \langle \alpha \rangle \big)$.
Since $W$ contains at least one generator of $H$, we assume that $\alpha \mid W$, and so there exists some $t \in [2,\ell]$ such that $\alpha \mid V_t$.
By the same argument as used before, \ref{itm:A1} guarantees that $V_t = V_{t_1} \bdot V_{t_2}$ for some non-trivial sequences $V_{t_1}, V_{t_2} \in \mathcal F \big( \langle \alpha \rangle \big)$ with $V_{t_1} \mid U_1$ and $V_{t_2} \mid U_{2} \bdot \ldots \bdot U_k$.
Since $|V_t| \le 4$, after renumbering if necessary, we suppose that $V_t \mid U_1 \bdot \ldots \bdot U_4$ in $\mathcal F (Q_8)$.
If we set $W' := U_1 \bdot \ldots \bdot U_4 \bdot V^{[-1]}_t$, then since $(V_1)_{Q_8 \setminus \langle \alpha \rangle} \mid U_1$, it is clear that $|\pi(W')| = 2$, and hence $V_t \mid U_1 \bdot \ldots \bdot U_4$ in $\mathcal B (Q_8)$, which, in view of \ref{itm:A1}, yields a contradiction.

We now consider the case $(V_1)_{Q_8 \setminus \langle \alpha \rangle} \mid U_1 \bdot U_2$.
Then, we must have that $|V_1| = |(V_1)_{Q_8 \setminus \langle \alpha \rangle}| = 4$ and each $U_i$ contains only two terms from $V_1$.
Moreover, $U_3, \ldots, U_k \in \mathcal A \big( \langle \alpha \rangle \big)$, and since $W$ contains at least one generator of $H$, we assume that $\alpha \mid W$.
Hence, there exists some $t \in [2,\ell]$ such that $\alpha \mid V_t$.
By the same argument as used before, \ref{itm:A1} guarantees that $V_t = V_{t_1} \bdot V_{t_2}$ for some non-trivial sequences $V_{t_1}, V_{t_2} \in \mathcal F \big( \langle \alpha \rangle \big)$ with $V_{t_1} \mid U_1 \bdot U_2$ and $V_{t_2} \mid U_{3} \bdot \ldots \bdot U_k$.
Suppose that $|V_t| = 4$, so that $V_t \in \{ \alpha^{[4]}, (\alpha^{-1})^{[4]} \}$.
If $|V_{t_1}| \ge 2$, then after renumbering if necessary, we suppose that $V_t \mid U_1 \bdot \ldots \bdot U_4$ in $\mathcal F (Q_8)$.
If we set $W' := U_1 \bdot \ldots \bdot U_4 \bdot V^{[-1]}_t$, then since $V_1 \mid U_1 \bdot U_2$, it follows that $|\pi(W')| = 2$, and thus $V_t \mid U_1 \bdot \ldots \bdot U_4$ in $\mathcal B (Q_8)$, which, in view of \ref{itm:A1}, leads to a contradiction.
If $|V_{t_1}| = 1$, then by renumbering if necessary, we may assume that $V_t \mid U_2 \bdot \ldots \bdot U_5$ in $\mathcal F (Q_8)$.
Since $U_2$ contains two terms from $Q_8 \setminus \langle \alpha \rangle$ and $\alpha^{[3]} \mid U_3 \bdot \ldots \bdot U_5$, it follows that $W' := U_2 \bdot \ldots \bdot U_5 \bdot V^{[-1]}_t$ must contain at least one $\alpha$ and two terms from $Q_8 \setminus \langle \alpha \rangle$.
Thus, $|\pi (W')| = 2$, and hence $V_t \mid U_2 \bdot \ldots \bdot U_5$ in $\mathcal B (Q_8)$, which, in view of \ref{itm:A1}, yields a contradiction.
Now, if $|V_t| \le 3$, then after renumbering if necessary, we suppose that $V_t \mid U_1 \bdot \ldots \bdot U_4$ in $\mathcal F (Q_8)$.
If we set $W' := U_1 \bdot \ldots \bdot U_4 \bdot V^{[-1]}_t$, then since $V_1 \mid U_1 \bdot U_2$, it follows that $|\pi(W')| = 2$, and thus $V_t \mid U_1 \bdot \ldots \bdot U_4$ in $\mathcal B (Q_8)$, which, in view of \ref{itm:A1}, leads to a contradiction.
\end{proof}

\smallskip
\noindent
{\bf Acknowledgement.} This paper was completed during the author's research visit to the AlgNTh group at University of Graz. The author is grateful to Alfred Geroldinger for his helpful comments on a preliminary version of this paper, for the excellent working conditions, and for all his hospitality.


\smallskip
\providecommand{\bysame}{\leavevmode\hbox to3em{\hrulefill}\thinspace}
\providecommand{\MR}{\relax\ifhmode\unskip\space\fi MR }
\providecommand{\MRhref}[2]{%
  \href{http://www.ams.org/mathscinet-getitem?mr=#1}{#2}
}
\providecommand{\href}[2]{#2}

\medskip

\end{document}